\documentclass[12pt,reqno]{amsart}
\usepackage{soul}
\usepackage{hyperref}
\usepackage{multicol,amsmath,graphics,cancel,soul,color,bbm,amsfonts,dsfont,tikz,mathrsfs,amssymb,mathtools,bm}
\usepackage[left=1in, right=1in, top=1.1in,bottom=1.1in]{geometry}
\usetikzlibrary{matrix,patterns,positioning}
\numberwithin{equation}{section}

\newcommand{\N}{\mathbb{N}}

\newcommand{\R}{\mathbb{R}}

\newcommand{\Indi}[1]{\mathbbm{1}_{#1}}

\newcommand{\Comb}[1]{\left(\begin{array}{c}#1\end{array}\right)}

\newcounter{dummy} \numberwithin{dummy}{section}
\newtheorem{Definition}[dummy]{Definition}

\newtheorem{Theorem}[dummy]{Theorem}

\newtheorem{Lemma}[dummy]{Lemma}

\newtheoremstyle{DefinitionStyle}  
	{9pt}				
	{9pt}				
	{}					
	{0pt}				
	{\bfseries}		    
	{.}					
	{3pt}				
	{}					
\theoremstyle{DefinitionStyle}
\newtheorem{definition}{Definition}

\newtheorem{lemma}{Lemma}
\newtheorem{proposition}{Proposition}
\newtheorem{theorem}{Theorem}
\newtheorem{corollary}{Corollary}
\newtheorem{remark}{Remark}

\def\1{{\rm l}\hskip -0.21truecm 1}

\newcommand{\vertiii}[1]{{\left\vert\kern-0.25ex\left\vert\kern-0.25ex\left\vert #1 
    \right\vert\kern-0.25ex\right\vert\kern-0.25ex\right\vert}}

\DeclareMathOperator\supp{Supp}

\def\R{\mathbb{R}}

\def\dilation{\mathfrak{D}}

\def\dh2l{\mathbf{d}_{\mathbb{H}_{2\ell}}}
\def\d2{\mathbf{d}_2}

\begin{document}

\title[Non-commutative Stein method]{Non-commutative Stein's Method: Applications to Free Probability and Sums of Non-commutative Variables}
\author{Mario D\'iaz\textsuperscript{\textdagger} and Arturo Jaramillo}
\address{Mario D\'iaz\textsuperscript{\textdagger}: Instituto de Investigaciones en Matem\'aticas Aplicadas y en Sistemas, Universidad Nacional Aut\'onoma de M\'exico, Mexico City, Mexico.}
\email{mario.diaz@sigma.iimas.unam.mx}
\address{Arturo Jaramillo: Centro de Investigaci\'on en Matem\'aticas, Jalisco S/N, Col. Valenciana 36023 Guanajuato, Gto.}
\email{jagil@cimat.mx}
\date{\today}

\thanks{\textsuperscript{\textdagger} Mario D\'iaz passed away during the preparation of this manuscript. This work is based on joint efforts with him, and it is dedicated to his memory.}

\begin{abstract}
We present a straightforward formulation of Stein's method for the semicircular distribution, specifically designed for the analysis of non-commutative random variables. Our approach employs a non-commutative version of Stein's heuristic, interpolating between the target and approximating distributions via the free Ornstein-Uhlenbeck semigroup. A key application of this work is to provide a new perspective for obtaining precise estimates of accuracy in the semicircular approximation for sums of weakly dependent variables, measured under the total variation metric. We leverage the simplicity of our arguments to achieve robust convergence results, including: (i) A Berry-Esseen theorem under the total variation distance and (ii) Enhancements in rates of decay under the non-commutative Wasserstein distance towards the semicircular distribution, given adequate high-order moment matching conditions. 

\bigskip
\noindent \textbf{MSC Classes:} 60F05, 46L53, 60G50, 60B10 
\end{abstract}

\maketitle

\section*{Dedication}
This work is dedicated to the memory of Mario D\'iaz, whose brilliant contributions and collaborative spirit were integral to this research. The majority of this manuscript was developed in close collaboration with him. While he was unable to review the final version, his vision and expertise shaped the foundation of this work. 


\section{Introduction}	
Let $\mu^{n}=\{\mu_{i,n}\ ;\ i\geq 1\}$ be a collection of centered probability measures with finite moments of order three and denote by $\sigma_{i,n}^2$ the variance of $\mu_{i,n}$. A fundamental problem in the theory of probability consists of studying the asymptotic properties of the inhomogeneous $n$-convolution
\begin{align}\label{eq:nundefintro}
\nu_{n}
  &:=\mu_{1,n}*\cdots*\mu_{n,n}.
\end{align}
This topic is of great mathematical relevance, particularly in the study of the law of large numbers, large deviations, and central limit theorems (CLTs), with the former topic being at the center of our discussion. Among the various approaches to CLTs, we adopt a probabilistic perspective based on Stein’s method, deliberately minimizing reliance on analytic tools. Our goal is to adapt these concepts to the framework of free independence, drawing parallels with classical theory, hereafter referred to as the ``tensorial regime". The development of this free Stein method will enable the transfer of constructions and theorems from tensorial theory to its free counterpart. Surprisingly, this transfer principle will produce more robust results in the domain of free, non-commutative variables than in the tensorial setting. As key applications of our theory, we examine the following:
\begin{enumerate}
\item[-] The quantification of accuracy in the inhomogeneous free Berry-Esseen theorem under the total variation metric for weakly dependent measures, which, in the homogeneous case, achieves a rate of order $1/\sqrt{n}$.
\item[-] A sharp enhancement in the inhomogeneous free Berry-Esseen theorem for summands with matching moments up to order $q$ under the non-commutative Wasserstein distance (see Sections \ref{freeberryessensec} and \ref{sec:Kantrubdistance} for definitions). In the homogeneous case, this result gives a rate of order $n^{-(q-1)/2}$.
\end{enumerate}
Beyond providing a robust, abstract description of the phenomenology arising when multiple small free random variables are added, the authors particularly appreciate this manuscript for presenting the argumentation in a manner that feels both friendly and natural to the classical tensorial perspective, further enhancing its already well-understood combinatorial reasoning based on the cumulant transform.\\
  
  \noindent Having anticipated the diversity of results to be expected after a thorough examination of our manuscript and aiming to keep the narrative simple at first,  we would like to invite the reader to regard the two corollary stated bellow, dealing with improvements in the free Berry-Esseen theorem, as our fundamental motivation, and realize the generalizations to weakly dependent measures as a natural extension, which, although slightly more technical, will follow the same line of reasoning. In the sequel, $\boxplus$ will denote the free convolution operation and $\mathbf{s}$ the standard semicircle distribution (see Section \ref{sec:noncommelementsofnon} for further details). For any probability distribution $\nu$, we will denote by $m_{k}[\nu]$ the $k$-th moment of $\nu$ and by $\mathfrak{g}[\nu]$, the unique semicircular distribution with the same mean and variance as $\nu$.  

\begin{corollary}[Special case of Theorem \ref{thm:inhomogeneousBerryEsseen}]\label{corollaryatthebeginning}
Let $ \mu_{k,n}$, with $k,n\geq 1$ be centered probability measures with uniformly bounded supports, such that 
\begin{align*}
\sum_{k=1}^{n}\int_{\R}x^2\mu_{k,n}(dx)=1,
\end{align*}
for all $n\geq 1$. Then, we can find a  constant $C>0$, independent of $n$, such that for large enough $n$, 
\begin{equation}\label{ineq:Berryesseenclassical}
d_{TV}(\mu_{1,n}\boxplus\cdots\boxplus \mu_{n,n},\mathbf{s}) \leq C\sum_{k=1}^nm_{3}[\mu_{k,n}],
\end{equation}
where $d_{TV}$ denotes distance in total variation. In particular, if the $\mu_{k,n}$ are constant over $k$, 
\begin{equation}\label{ineq:Berryesseenclassical2}
d_{TV}(\mu_{1,n}\boxplus\cdots\boxplus \mu_{n,n},\mathbf{s}) \leq Cn^{-1/2}.
\end{equation}
\end{corollary}  
Further improvements in the rate can be achieved by changing the metric $d_{TV}$ by the non-commutative 1-Wasserstein distance $d_{W}$ and imposing a suitable restriction over the moments of the measures $\mu_{i,j}$.  
\begin{corollary}[Special case of Theorem  
\ref{thm:inhomogeneousBerryEsseensecond}]\label{corollaryatthebeginning}
Let $ \mu_{k,n}$, with $k,n\geq 1$ be centered probability measures with uniformly bounded supports, such that 
\begin{align*}
\sum_{k=1}^{n}\int_{\R}x^2\mu_{k,n}(dx)=1,
\end{align*}
for all $n\geq 1$. Denote by $q\geq 1$ be the largest integer such that $m_{j}[\nu]=m_{j}[\mathfrak{g}[\nu]]$, for all $j\leq q$. Then, we can find a  constant $C>0$, independent of $n$, such that for large enough $n$, 
\begin{equation}\label{ineq:Berryesseenclassical12}
d_{W}(\mu_{1,n}\boxplus\cdots\boxplus \mu_{n,n},\mathbf{s}) \leq C\sum_{k=1}^nm_{q+1}[\mu_{k,n}].
\end{equation}
In particular, if the $\mu_{k,n}$ are constant over $k$, 
\begin{equation}\label{ineq:Berryesseenclassical22}
d_{W}(\mu_{1,n}\boxplus\cdots\boxplus \mu_{n,n},\mathbf{s}) \leq Cn^{-(q-1)/2}.
\end{equation}
\end{corollary}  

\noindent While not as general as the formulation presented in Section 5, this result still provides a solid description of the behavior of large convolutions of small measures. It goes beyond what is  currently available in the literature on limit theorems for non-commutative random variables, though it is important to note that several related works have laid the groundwork for this  (see \cite{MR2370598}, \cite{MR3101842}, \cite{MR2370598}, \cite{MR4607696}).

\begin{remark}
Relation \eqref{ineq:Berryesseenclassical12} provides a quantization of the idea that a large superposition of measures should become closer to the semicircular distribution if the $\mu_{n,k}$'s become ``nearer to semicircular''. This notion of ``nearness'' is codified by number $q$ of moments that match the moments of the semicircular distribution.       	
\end{remark}
The organization of the paper emphasizes the parallelism of tensorial and free argumentations, as this practice will shed light on what type of results one should expect to hold, just by looking at the, already very well developed, theory of tensorial limit theorems. The precise structure of the manuscript is as follows. We begin with a brief overview of the development of the theory of Stein's method, deferring the presentation of our applications to later sections. First, we address Gaussian approximations in the tensorial regime, focusing on the application-based historical progression without emphasizing sharpness or generality (Section \ref{subsec:freeclt}). We then turn to the free case (Section 1.2), and finally, revisit the free case but now with the incorporation of Stein’s method (Section \ref{subsec:tensCLTviastein}). In Section \ref{sec:prelim} we present the preliminaries, among which we include a discussion of the basic ideas from Stein's method, as well as a brief introduction to our main ideas. In Section \ref{sec:steinmethodlologyNC}, we develop our version of non-commutative Stein method in full detail. Our main technical result is presented in Section \ref{freeberryessensec}, which is followed by the corresponding proofs, described in   Section \ref{sec:proofofmainresults}. Section \ref{sec:technicallemas} is utilized for proving some technical lemmas that are used throughout the paper.\\

\section{Literature Review}
In this section, we present a brief summary of the main developments regarding tensorial central limit theorems.
\subsection{Measurements of accuracy in the CLT}\label{subsec:MeasaccCLT}
Let $\mathcal{P}(\R)$ be the set of probability measures over the real line and $\dilation_{r}:\mathcal{P}(\R)\rightarrow \mathcal{P}(\R)$, the dilation operator by the non-zero real factor $r$, defined through $\dilation_{r}[\mu][A]:=\mu[A/r]$. The celebrated Lindeberg central limit theorem (see \cite{MR1544569} and \cite[Page 307]{MR1810041}) establishes minimal conditions under which $\dilation_{1/s_{n}}[\nu_n]$, with  
$$s_{n}^2:=\sum_{k=1}^n\int_{\R}x^2\mu_{k,n}(dx),$$ converges weakly towards the standard Gaussian measure 
$$\gamma(dx) \coloneqq \frac{1}{\sqrt{2\pi}}e^{-\frac{1}{2}x^2}.$$ Implementations of the Gaussian approximation for $\dilation_{1/s_{n}}[\nu_n]$ often require accurate assessments of the associated error, measured according to a suitable probability distance. Andrew Berry and Carl-Gustav Esseen addressed this topic in \cite{MR0011909} and \cite{MR3498}, and proved the existence of a constant $C>0$, independent of $n$, such that 
\begin{equation}
\label{eq:BEclassic}
    d_{\textnormal{Kol}}(\dilation_{1/s_{n}}[\nu_n],\gamma) \leq C s_{n}^{-3}\sum_{k=1}^n\int_{\R}|x|^{3}\mu_{k,n}(dx),
\end{equation}
where  $d_{\textnormal{Kol}}$ is the Kolmogorov distance 
\begin{equation*}
    d_{\textnormal{Kol}}(\mu,\nu) = \sup_{z\in\mathbb{R}} \lvert \mu(-\infty,z] - \nu(-\infty,z] \rvert.
\end{equation*}
After the publication of this result, many improvements and developments were made, among which we remark the influential manuscript \cite{ChStein} published by Charles Stein in 1972, which constitutes the foundations of the perspective taken in this paper. This work marked the beginning of the so-called  \emph{Stein's method},  which nowadays refers to a collection of techniques that allow bounding probability distances by means of functional operators. The core of the theory stems from the fact that the left-hand side of \eqref{eq:BEclassic} is obtained as the supremum over functions of the form $h_z(x)=\Indi{\{x\leq z\}}$, with $z$ ranging over the real line, of the expression
\begin{align}\label{eqsteinintuition}
\int_{\R}h_z(x)(\dilation_{1/s_{n}}[\nu_{n}]-\gamma)(dx).
\end{align}
Stein's method regards  \eqref{eqsteinintuition} as an expression of the form
\begin{align}\label{eq:SteinIBPclassic}
    \int_{\R} (xDf_z(x)-D^2f_z(x))\dilation_{1/s_{n}}[\nu_n](dx),
\end{align}
where $D$ denotes the derivative operator and $f_z$ is a function depending on $z$, and satisfying adequate smoothness properties. The term \eqref{eq:SteinIBPclassic} has the advantage of not involving the substraction by the measure $\gamma$, appearing in \eqref{eqsteinintuition}, as this operation becomes codified in the definition of $f_z$. This perspective, combined with the classical Lindeberg argument, yields a simple proceedure for recovering the results by  Berry and Esseen. Further details about this technique will be outlined in Section \ref{Section:ClassicalSteinMethod}. The interested reader can consult \cite{ChenGoldShao} for a thorough discussion of the topic from a classical point of view and \cite{NoPe} for a perspective oriented to its application to Gaussian functionals. A remarkable advantage of Stein's method is its versatility for being applicable to  sums of non-necessarily independent random variables. Indeed, the combination of Stein's method with tools such as exchangeable pairs or Malliavin calculus has had great success in the theory of Gaussian approximations (see \cite{ChenGoldShao}, \cite{NoPe}). One of the most straightforward applications of the above ideas is the Gaussian approximation for sums of weakly dependent variables, with dependency encoded by ``dependency graphs'', a type of probabilistic structure which applies in a wide variety of models, including time series and random geometric graphs (see \cite{MR1986198}).\\

\noindent We would like to draw the reader’s attention to the manuscripts \cite{MR3230000} and \cite{MR3112923}, which discuss the quality of the entropic version of the CLT, as well as its counterpart in the total variation metric:
\begin{align*}
d_{TV}(\mu,\nu)
&:=\sup_{A\in\mathcal{B}(\R)}|\mu[A]-\nu[A]|,
\end{align*}
under regularity assumptions on the corresponding summands. These works are among the few general results available in the literature concerning asymptotic Gaussianity for normalized inhomogeneous convolutions measured in total variation.\footnote{Although the total variation metric is not typically within the reach of limit theorems for models of this kind, results concerning the Kolmogorov and Wasserstein metrics have been studied quite broadly.} It is worth noting that the main technical hypothesis in \cite{MR3230000} involves a uniform boundedness of the entropic distance to normality. As we will discuss in detail in the forthcoming Section \ref{freeberryessensec}, this condition can be replaced in the non-commutative free setting by a boundedness condition on the summands. We regard this as an instance of the aforementioned improvement in robustness when transitioning from the tensorial to the free setting.

\subsection{Measurements of accuracy in the free CLT}\label{subsec:freeclt}
The development of the above ideas in the framework of non-commutative random variables is a topic that has shown to be of great relevance, due to its applications in operator algebras, random matrices, combinatorics, and representation theory of symmetric groups, among others. Our main focus will be on the case of variables exhibiting weakly dependence in the free sense  (see Section \ref{sec:noncommelementsofnon}). The notion of freeness  induces a natural convolution obtained as the distribution of the sum of freely independent variables with prescribed marginals, which in turn gives sense to the free version of \eqref{eq:nundefintro}. The central limit theorem has a counterpart in the realm of free-convolution, with the semicircular law 
$$\mathbf{s}(dx):=\frac{2}{\pi}\sqrt{1-x^2}dx$$ 
playing the role of the Gaussian distribution. Many efforts have already been put on the investigation of these types of limit theorems. In particular, the free version of \eqref{eq:BEclassic} has been studied by Christyakov and G\"otze in \cite{MR2370598}, where they showed that for freely independent non-commutative standardized random variables $X_1,\dots,X_n$ with distribution $\mu$, satisfying mild conditions, the Kolmogorov distance between the law of $n^{-1/2}(X_1+\cdots+X_{n})$ and $\mathbf{s} $ is bounded by a constant multiple of $n^{-1/2}.$ This result relies on the interplay between free convolution and the reciprocal of the Cauchy transform for the probability measures under consideration, a technique that exploits the combinatorial relations encoded in the notion of freeness. Although quite different in nature, the use of the reciprocal of the Cauchy transform does draw a parallelism with the Fourier transform in the classical central limit theorem, since both concepts translate the nature of the convolution operation into complex functions satisfying adequate additivity properties.\\

\noindent Quantitative estimates of the free central limit theorem in Kolmogorov distance were studied in \cite{MR2370598}, and enhanced to an Edgeworth type expansion in \cite{MR3101842}. The paper \cite{MR2370598} finds Berry Esseen type bounds as well with the inhomogeneous case, but their result in this direction, when specialized to the homogeneous case, yield a quadratically smaller rate of convergence when compared with the free homogeneous Berry Esseen bound available in the same paper. Finally, we would like to mention the manuscript \cite{MR4607696}, which exhibits an interesting phenomenology without a parallel in classical convolution: the fact that a vanishing third moment condition improves the rate of convergence in the free CLT from the order $1/\sqrt{n}$ to order $ 1/n$.\\

\subsection{Inhomogeneous free Berry-Esseen theorem via Stein's method}\label{subsec:tensCLTviastein} Taking into consideration that the  development of classical limit theorems started from a direct computation approach, then drifted to a Fourier perspective and was subsequently strengthened by the introduction of Stein's method, it is quite natural to wonder if a similar methodology could be also conceived in the non-commutative framework in a spirit similar to the tensorial Stein's method \cite{ChStein}.  We will show that this is indeed the case, and that summands can be allowed to only satisfy weak dependency (in the free sense) in place of full freeness. As a non-negligible amount of notation is required for the formulation of our main result, we pospone its presentation to Section \ref{freeberryessensec}. A pivotal element within our proofs, is the super-convergence of the $\nu_n$'s, which Bercovici and Voiculescu showed to be valid for inhomogeneous convolutions of uniformly bounded probability measures (see \cite{MR1355057}). This property guarantees that under the condition of uniform boundedness of the support of the $\mu_{k,n}$, the support of $\nu_n$ is contained in the interval $[-3,3]$, for $n$ sufficiently large. The generality of these results allows us to by-pass some of the technical complications.\\

\noindent\textit{Related work}\\
We would like to remark that the idea of implementing Stein's method techniques in the context of non-commutative random variables has already had remarkable advances at the time this manuscript was written. Among them, we emphasize the work by G\"otze and  Tikhomirov \cite{MR2171668}, where a differential equation for characterizing the semicircular distribution is proposed, and a variety of applications in random matrices were addressed with this technique. In \cite{MR2978133}, a free Stein  methodology was adapted to the context of functionals of the free Brownian, with a perspective of non-commutative Malliavin calculus. This led to the formulation of the free fourth moment theorem, that has shown to be a formidable tool in the study of limit theorems for additive functionals (see \cite{MR3848230}, \cite{MR3548767}, \cite{MR3378462}, \cite{MR3217054}). The techniques from \cite{MR2978133} were refined by Cebron in \cite{MR4203499}, where a quantitative version of the fourth moment theorem in the Wasserstein distance was proved by a combination of semigroup arguments, free Malliavin calculus and free Stein discrepancy. Regardless of these advances, the topic is far from being complete, as its most natural application: a simple proof of the free central limit theorem with a perspective parallel to \cite{ChStein}, remains an open problem. This manuscript proposes a particular type of Stein's method, fundamentally different from \cite{MR2171668}, which allows for a simple implementation to free convolutions of probability measures, yielding a Berry-Esseen type bound of the type \eqref{ineq:Berryesseenclassical} for non-necessarily identically distributed and non-necessarily free, random variables.\\

\noindent We would like to comment on the similarities and differences of our work in comparison to the references mentioned above: the work \cite{MR4203499} is, in the authors' opinion, the closest one to this manuscript in spirit, due to the fact that it utilizes the free Ornstein Uhlenbeck semigroups as pivotal tool for interpolating the probability measures under consideration. Therein, the relation between the infinitesimal change in time of the semigroup and the non-commutative derivative, previously observed in  \cite{MR2978133}, is exploited to obtain quantitative assessments of the rate of convergence in the free fourth moment theorem. The paper \cite{MR2978133} has as well some of these ideas luring behind some of the arguments, although in a less explicit manner. Both of these pieces of work are of great influence to our work and they do have certain similarities with ours, especially at the level of technical computations. However, the nature of the problems addressed in \cite{MR4203499}  and \cite{MR2978133}, as well as the solutions themselves are entirely different from our main result: the overall theme discussed in \cite{MR4203499}  and \cite{MR2978133} was the study of functionals of the free Brownian motion via Malliavin calculus, while ours is the study of dilated sums of self-adjoint free random variables, via arguments inspired in the generator approach from the theory of classical Stein's method. In terms of the statement of our main  result, the closest manuscripts to our work are  \cite{MR2171668}, \cite{MR4607696}, \cite{MR3230000}, \cite{MR3112923} and \cite{MR4607696}, where the free Berry-Esseen theorem for free   random variables with improvement under vanishing third moments, and Edgeworth expansion estimates is proved using cumulant transforms. As previously mentioned, our main application aligns with the theme of these papers, but our perspectives deviate substantially: ours is mostly probabilistic, while the ones currently available are analytical, and there is no clear connection yet of how the identities from the theory of free Stein method can be captured by Cauchy, cumulant or $\mathcal{R}$-transforms.\\

\noindent Finally, we would like to mention two pieces of work which, although not directly related to the free Berry-Esseen theorem, possess common features with the manuscripts cited thus far: (i) the paper \cite{ArBanTseng} by Arizmendi et.al., where the Lindeberg method is successfully applied in the context of boolean and monotone convolutions. The classical Lindeberg method is known to have some common points in their argumentation in comparison with Stein's method; mainly the treatment of the effect of removing one of the independent components under consideration via Taylor expansions. The non-commutative Lindeberg method discussed in \cite{ArBanTseng} and further developed in our paper extends the classical parallelism into the framework of free convolution. Additionally, it is worth highlighting the contributions by Goldstein and Kemp \cite{GoldKemp}, as well as C\'ebron, Fathi, and Mai \cite{cebron2020note}, where the concept of the free bias transform is explored. Their work provides a new perspective on characterizing infinitely divisible distributions within free probability theory.\\

\section{Preliminaries}\label{sec:prelim}
\subsection{Elements of non-commutative probability}\label{sec:noncommelementsofnon}
In this section we recall some basic notions from free probability.\\ 

\noindent \textit{Non-Commutative Probability Spaces}\\
Let $\mathcal{A}$ be a unital $C^*$-algebra and $\tau:\mathcal{A}\to\mathbb{C}$ a positive unital linear functional. We then say that the pair $(\mathcal{A},\tau)$ is a $C^*$\textit{-probability space}. The elements of $\mathcal{A}$ will be called non-commutative random variables. An element $a\in\mathcal{A}$ will be called self-adjoint if it satisfies $a=a^*$. The functional $\tau$ serves as the non-commutative analogue of the expectation operator in classical probability. In this spirit, for a given collection $a_1,\dots,a_k$ of elements in $\mathcal{A},$ we will refer to the values of $\tau[a_{i_1}\cdots a_{i_n}]$, for $1\leq i_1,...,i_{n}\leq k$, $n\geq1$, as the \textit{mixed moments} of $a_1,\dots,a_k$.\\

\noindent The  characterization of a non-commutative random variable via the description of its moments (procedure that is purely algebraic) can be enhanced to the familiar notion of describing the distribution of a random variable by means of a probability measure, provided that the variable under consideration is self adjoint. More precisely, if $a\in\mathcal{A}$ is self-adjoint, then there exists a unique probability measure $\mu_a$, typically referred as its ``analytical distribution'' such that 
\begin{align*}
\int_{\R}x^{k}\mu_a (dx)
  &=\tau [a^{k}], \quad \text{ for } k\in \N.
\end{align*}

\noindent \textit{Free independence}\\
In general, knowledge of the individual distribution of two given self-adjoint elements $a,b\in \mathcal{A}$ says very little their joint distribution (mixed moments), unless some  additional structure on the relation between $a$ and $b$ is imposed. A possible venue for addressing this topic, consists of making use of the notion of free independence, which we describe next.

\begin{definition} Let $\{A_n\}_{n\geq 1}$ be a sequence of subalgebras of $\mathcal{A}$. For $a\in \mathcal{A}$, denote the centering of $a$ by $\bar{a}:=a-\tau(a)$. We say that $\{A_n\}_{n\geq 1}$ are freely independent, or free, if
\begin{equation}\label{eq:tauoveralphasfreeness}
\tau[\bar{a}_1\bar{a}_2 \cdots \bar{a}_k]=0,
\end{equation}
for every choice of $a_1,\dots, a_k\in \mathcal{A}$ such that  $a_i\in A_{j(i)}$, with $j(1),\dots, j(k)$ satisfying $j(i)\neq j(i+1)$.
\end{definition} 
\noindent The notion of free independence is particularly useful when applied to the algebras generated by two elements in $a,b\in\mathcal{A}$, as it allows us to recover the joint moments of $a,b$ in terms of their individual distributions of $a$ and $b$ respectively. In particular, one can check that when $a_0,a_1,a_2\in\langle a\rangle$ and $b_0,b_1,b_2\in\langle b\rangle$, then, provided that the algebras generated by $a,b$ (denoted by $\langle a\rangle$ and $\langle b\rangle$), are free, we have the identities 
\begin{align}
\tau[a_0b_0]&=\tau[a_0]\tau[b_0]\nonumber\\
\tau[a_1b_0a_2] &=\tau[a_1a_2]\tau[b_0]\nonumber\\
\tau[a_1b_1a_2b_2] &=\tau[a_1a_2]\tau[b_1]\tau[b_2]+\tau[a_1]\tau[a_2]\tau[b_1b_2]-\tau[a_1]\tau[b_1]\tau[a_2]\tau[b_2].\label{eq:firstmomentsfreeness}
\end{align}

\subsubsection{Combinatorics of mixed moments}
A pivotal part of our proof relies on the computational power of the so-called free cumulants, which we define next. Proposition \ref{prop:krewerasprop} bellow, which gives an explicit formula for the mixed moments of free variables. To adequately formulate this result, we first introduce some basic combinatorics tools. In the sequel, for $n\in\N$, we define $[n]:=\{1,\dots, n\}$. For a given totally ordered finite set $X$, we denote by $\mathfrak{P}(X)$ the set of partitions of $X$. We will simply write $\mathfrak{P}(n)$ when $X=[n]$.

\begin{definition}
Define the function $\tau_n:\mathcal{A}^n\rightarrow\mathbb{C}$ via
\begin{align*}
\tau_n[a_1,\dots, a_n]
  &:=\tau[a_1\cdots a_n]	.
\end{align*}
For a given $V\subset [n]$ of the form $V=\{i_1,\dots, i_r\}$, with $i_j<i_{j+1}$, we define 
\begin{align*}
\tau_V[a_1,\dots, a_n]
  &:=\tau_n[a_{i_1},\dots, a_{i_n}].
\end{align*}
For a partition $\pi\in \mathfrak{P}(n)$, we define 
\begin{align*}
\tau_{\pi}[a_1,\dots, a_n]
  &:=\prod_{V\in \pi}\tau_{V}[a_1,\dots, a_n].	
\end{align*}
\end{definition}
The set $\mathfrak{P}(n)$ is a poset with lattice structure defined through refinement. The associated order will be denoted by ``$\leq $''. The number of elements in a partition $\pi\in\mathfrak{P}(n)$ will be denoted by $|\pi|$. A given element $\pi\in\mathfrak{P}(n)$ is said to have a crossing if there are blocks $V_1,V_2\in\pi$, and different elements $a,b\in V_1$ and $c,d\in V_2$, such that $a<c<b<d$. The partition $\pi$ is said to be non-crossing if no blocks of this type can be found. The set of non-crossing partitions, ordered according to ``$\leq $'' becomes a lattice, which in the sequel will be denoted by $NC(n)$. Denote by $\hat{0}$ and $\hat{1}$ the minimal (respectively maximal) partitions $\hat{0}:=\{\{j\}\ ;\ j\in [n]\}$ and $\hat{1}:=\{[n]\}$. The M\"obius function $\mu$ associated to the lattice $NC(n)$ is known to be multiplicative and satisfy 
\begin{align*}
\mu(\hat{0},\hat{1})
  &:=\frac{(-1)^{n-1}}{n+1}\Comb{2n\\n}.	
\end{align*}
\begin{definition}
Consider additional ordered symbols $\bar{1},\dots, \bar{n}$ and interlace them with the elements of $[n]$ in the following alternating way:
$$1\bar{1}2\bar{2}\cdots n\bar{n}.$$	
Let $\pi$ be an element of $NC(n)$. Then its Kreweras complement $K[\pi]$ is the maximal element $\sigma\in NC(\{\bar{1},\dots, \bar{n}\})\cong NC(n)$ satisfying 
\begin{align*}
\pi\cup\sigma\in NC(\{1,\bar{1},2,\bar{2},\cdots, n,\bar{n}\}).
\end{align*}
\end{definition}

\noindent Next we introduce the notion of free cumulants
\begin{definition}
Let $a_1,\dots, a_n\in\mathcal{A}$ be given. The free cumulants $\{\kappa_{\pi}[a_1,\dots, a_n]\ ;\pi\in\mathcal{P}\}$	are defined as 
\begin{align*}
\kappa_{\pi}[a_1,\dots, a_n]
  &:=\sum_{\substack{\sigma\in NC(n)\\\sigma\leq \pi}}\tau_{\sigma}[a_1,\dots, a_{n}]\mu(\sigma,\pi).
\end{align*}
\end{definition}
The cumulants are known to be a pivotal element in the understanding of the distribution of free non-commutative random variables. We will not delve too much into the topic, as the main piece of information that we will require is Proposition \ref{prop:krewerasprop} bellow. We refer the reader to \cite{MR2266879}. The following result is the anticipated formula for the computation of mixed moments. Its proof can be found in \cite[Theorem 14.4.]{MR2266879}
\begin{proposition}\label{prop:krewerasprop}
	Consider random variables $a_1,\dots,a_n,b_1,\dots,b_n \in \mathcal{A}$  such that ${a_1,\dots,a_n}$ and ${b_1,\dots,b_n}$ are freely independent. Then we have
\begin{align*}
\tau[a_1b_1\cdots a_n b_n]
  &=\sum_{\pi\in NC(n)}\kappa_{\pi}[a_1,\dots, a_n]\tau_{K(\pi)}[b_1,\dots, b_b]	
\end{align*}
\end{proposition}

\subsubsection{Free convolution}
The notion of free convolution, conceived as the analytic distribution of the sum of two free random variables with prescribed probability law, was defined in \cite{MR799593} for probability measures with compact support and later extended in \cite{MR1165862} for the case of finite variance, and in \cite{MR1254116} for the general unbounded case. The general definition relies on properties of the Cauchy transform of $\mu \in \mathcal{M}$, but for purposes of the current paper, we will simply focus on its algebraic definition. 
\begin{Definition}
Given two probability measures $\mu$ and $\nu$, we construct a non-commutative probability space $(\mathcal{A},\tau)$ and self-adjoint elements $a,b\in\mathcal{A}$, with analytic distributions $\mu$ and $\nu$, respectively, such that $a$ and $b$ are free. In this setting, the analytic distribution of $a+b$ is called free additive convolution and is denoted by $\mu\boxplus\nu.$ 
\end{Definition}

\subsection{Non-commutative Kantorovich-Rubenstein distance}\label{sec:Kantrubdistance}
In this section, we introduce the notion of the non-commutative Kantorovich-Rubenstein distance, first presented by Biane and Voiculescu in \cite{MR1878316}. To this end, let $\mathfrak{J}$ denote the set of states over $\mathcal{A}^{2}$. We refer to the first and second components in $\mathcal{A}^{2}$ as $X$ and $Y$, respectively. For a given pair of probability measures $\gamma_1, \gamma_2 \in \mathcal{P}(\R)$, we define $\Pi[\gamma_1, \gamma_2]$ as the set of elements in $\mathfrak{J}$ whose marginals over the first and second variables correspond to $\gamma_1$ and $\gamma_2$, respectively.

\begin{definition}
The $1$-Kantorovich-Rubenstein metric is defined by 
\begin{align*}
d_{W}(\gamma_1, \gamma_2)
  &:= \inf_{\tau \in \Pi[\gamma_1, \gamma_2]} \tau[|Y - X|].
\end{align*}
\end{definition}

This definition is inspired by its counterpart in optimal transport, given by
\begin{align*}
\mathbbm{d}_{W}(\gamma_1, \gamma_2)
  &= \inf_{\pi \in \overline{\Pi}[\gamma_1, \gamma_2]} \int_{\R^{2}} |x - y|^{p} \pi(dx, dy),
\end{align*}
where $\overline{\Pi}[\gamma_1, \gamma_2]$ denotes the set of tensor transport plans in $\R^{2}$. These are probability measures over $\R^{2}$ whose marginals over the first and second components are $\gamma_1$ and $\gamma_2$, respectively. Since classical probability spaces are particular instances of non-commutative ones, the following inequality holds:
\begin{align*}
d_{W}(\gamma_1, \gamma_2)
  &\leq \mathbbm{d}_{W}(\gamma_1, \gamma_2).
\end{align*}

The following result can be verified directly from the definition of $d_{W}$. For a detailed proof, we refer the reader to \cite{JaVa}.

\begin{Lemma}\label{Lemma:convolutioninequality}
If $\gamma_1, \rho_1, \dots, \gamma_n, \rho_n \in \mathcal{P}(\R)$ are probability measures with finite first moments, then 
\begin{align}\label{eq:keyineq1}
d_{W}(\gamma_1 \boxplus \cdots \boxplus \gamma_n, \rho_1 \boxplus \cdots \boxplus \rho_n)
	&\leq \sum_{k=1}^{n} d_{W}(\gamma_k, \rho_k).
\end{align}
\end{Lemma}
\subsection{Free central limit theorems}
The free convolution naturally raises the question of whether a normalized large sum of freely independent random variables posseses a non-trivial limiting distribution, as in the classical case. The following theorem gives an answer to this question
\begin{Theorem}[Free central limit theorem]\label{eq:freclt}
Let $\{a_{k}\}_{k\geq 1}$ be a collection of self-adjoint free random variables defined on  $(\mathcal{A},\tau)$, with $a_{k}$ having analytic distribution $\mu_{k}$. If the $a_{k}$ are identically distributed and standardized, with finite moments of order $3$, then the analytic distribution of 
\begin{align}\label{eq:normalizedSn}
    \frac{1}{\sqrt{n}}\sum_{k=1}^{n}a_{k},
\end{align}
converges weakly towards the semicircle distribution $\mathbf{s}$.
\end{Theorem}
The free Berry-Esseen theorem below, as presented in \cite[Proposition 2.6]{MR2370598}, describes the rate of convergence of the theorem above
\begin{theorem}[Chistyakov, G\"otze]
Let the notation of Theorem \ref{eq:freclt} prevail. Let $\mu$ denote the common distribution of the $a_k$ and $\nu_{n}$  the analytic distribution of \eqref{eq:normalizedSn}. Then we have that 
\begin{align*}
d_{K}(\nu_n,\mathbf{s})
  &\leq \frac{C}{\sqrt{n}}\left(\int_{\R}|x|^3\mu(dx)+\left(\int_{\R}|x|^4\mu(dx)\right)^{1/2}\right),
 \end{align*}
for some universal constant $C>0$.
\end{theorem}

\subsection{Classical Stein's Method Revisited}
\label{Section:ClassicalSteinMethod}
The purpose of this section is to dissect the fundamental pieces of Stein's method, in order to establish an adequate generalization that could serve in the framework of non-commutative probability.\\

\noindent \textit{The Stein Heuristic}\\
In what follows, for every probability measure $\mu$ defined in $\R$ and every function $\psi:\R\rightarrow\R$ integrable with respect to $\mu$, we denote the action of $\mu$ over $\psi$ by $\langle\mu,\psi\rangle$,  namely, 
\begin{align*}
\langle\mu,\psi\rangle
  &:=\int_{\R}\psi(x)\mu(dx).
\end{align*}
If $\mathcal{K}$ is a rich enough collection of measurable functions, whose elements are integrable with respect to $\mu$ and $\gamma$ denotes the standard Gaussian distribution, then the identity $\mu=\gamma$ is equivalent to 
\begin{align}\label{eq:idenmugammanai}
\langle\mu,\psi\rangle-\langle\gamma,\psi\rangle=0,
\end{align}
for all $\psi\in\mathcal{K}.$ In most applications, a direct analysis of $\langle\mu,\psi\rangle-\langle\gamma,\psi\rangle$ is difficult to carry, making it quite attractive to find characterizations of the Gaussian distribution that could serve as alternative to \eqref{eq:idenmugammanai}. The following lemma provides a very powerful equivalence of this sort. In the sequel,  $\iota:\R\rightarrow\R$ will denote the identity function $\iota(x)=x$ and  $\mathcal{C}^{\ell}(\R^{l};\R^d)$  the set of $\R^d$ valued, $\ell$-times continuously differentiable functions defined in $\mathbb{R}^{l}$.

\begin{lemma}[Stein's lemma]
Suppose that $\mu$ is a probability measure over $\R$ such that for every  $f\in \mathcal{C}^{2}(\R;\R)$ with derivatives integrable with respect to $\mu$, 
\begin{align}\label{eq:Steincharact}
\langle \mu, \iota\cdot Df-D^2f\rangle=0.
\end{align}
Then $\mu$ is equal to the standard Gaussian distribution $\gamma$.
\end{lemma} 
Identity \eqref{eq:Steincharact} sets the foundations of the so called ``Stein's Heuristics'', which suggests that   \eqref{eq:idenmugammanai} is close to zero when \eqref{eq:Steincharact} is. The formalization of these ideas is implemented through the Stein equation, which we discuss next. \\

\noindent \textit{The Stein Equation}\\
Recall that $\mathcal{P}(\R)$ denotes the set of probability measures over $\R$ and denote by $\mathcal{P}_{\ell}(\R)$ the set of elements in $\mathcal{P}(\R)$ with finite absolute moments of order $\ell$. For a given element $\mu\in\mathcal{P}_{\ell}(\R)$, the quantity $m_{\ell}[\mu]$   denotes the moment of order $\ell$ of $\mu$. Let $\mathcal{K}$ be a symmetric subset of $\mathcal{C}^{2}(\R,\R)$ containing the real and imaginary parts of $e^{\mathbf{i}\xi x}$, for $\xi\in\R$, we can define the distance
\begin{align}\label{dKdef}
d_{\mathcal{K}}(\mu,\nu)
  &=\sup_{f\in\mathcal{K}}|\langle f,\mu-\nu\rangle|	.
\end{align}
Consider as well the operator $\mathcal{L}$, defined through 
\begin{align}\label{eq:eqLdef}
\mathcal{L}[g](x)
  &:=-xg(x)+Dg(x).
\end{align}
Then, by Stein's lemma, the validity of
\begin{align}\label{eq:Steinheuristic}
\langle \mu,\mathcal{L}[Df]\rangle=0,
\end{align}
for $f$ twice differentiable, satisfying $Df\in \mathcal{K}$ integrable with respect to $\mu$,  implies that $\mu$ is the standard Gaussian distribution. For a given element $h\in \mathcal{C}^{0}(\R,\R)$, consider the Stein equation 
\begin{align}\label{eq:Steineq}
\mathcal{L}[Df](x)=h(x)-\langle \gamma,h\rangle,
\end{align}
where $\gamma$ denotes the standard Gaussian distribution and $h$ is assumed to be integrable with respect to $\gamma$. Assuming we can guarantee the existence of a solution to this equation, by first integrating both sides of \eqref{eq:Steineq} with respect to $\mu$, and then taking supremum over $h\in \mathcal{K}$, we obtain
\begin{align*}
\sup_{f\in \mathcal{S}[\mathcal{K}]}\langle\mu,\mathcal{L}[Df]\rangle
=d_{\mathcal{K}}(\gamma,\mu),
\end{align*}
where $\mathcal{S}[\mathcal{K}]$ denotes the set of functions $f$ obtained as the solution to \eqref{eq:Steineq}. The problem of describing the proximity of $d_{\mathcal{K}}(\gamma,\mu)$ around zero thus reduces to finding upper bounds for the expression 
\begin{align*}
\langle\mu,\mathcal{L}[Df]\rangle,
\ \ \ 
\text{ for } \ f\in\mathcal{S}[\mathcal{K}].
\end{align*}
Typically, the dependence of $\langle\mu,\mathcal{L}[Df]\rangle$ over the underlying function $h$ that characterizes $f$ through \eqref{eq:Steineq}, is removed from the analysis by proving that $\mathcal{S}[\mathcal{K}]$ is contained in a larger, but easier to describe, set $\mathcal{X}$; so that the supremum of $\langle\mu,\mathcal{L}[Df]\rangle$ over $\mathcal{S}[\mathcal{K}]$ is bounded by the supremum over $\mathcal{X}$.\\

\noindent \textit{The semigroup approach for solving Stein's equation}\\
An important problem that gets hidden in the aforementioned program is the description of the properties of the solution $f$ to the Stein equation \eqref{eq:Steincharact}. How do we know there even exists a solution and more importantly, how do we know that it will lie within a reasonably easy to describe set $\mathcal{X}$? One way of dealing with these questions is by the implementation of the so-called semigroup approach, which we describe next: Let $P_{\theta}$, for $\theta\geq0$, denote the collection of operators defined on the space of functions integrable with respect to the push forward of any affine transformation of $\gamma$, taking the form
\begin{align*}
P_{\theta}[h](x)
  &:=\int_{\mathbb{R}}h(e^{-\theta}x+\sqrt{1-e^{-2\theta}}y)\gamma(dy).
\end{align*}
It is straightforward to check that $\{P_{\theta}\}_{\theta\geq 0}$ is a semigroup and that its generator is well defined over the set of twice differentiable functions with second derivative having polynomial growth. Its value coincides with the operator $\mathcal{L}\circ D$, with $\mathcal{L}$ defined by \eqref{eq:eqLdef}, for functions $f$ such that  $Df$ belongs to the domain of $\mathcal{L}$. Moreover, we have that $P_{\infty}[h](x):=\lim_{\theta\rightarrow\infty}P_{\theta}[h](x)=\langle \gamma,h\rangle$, while $P_{0}[h](x)=h(x)$, so that
\begin{align}\label{interpolationclassic}
\langle \gamma,h\rangle-h(x)=P_{\infty}[h](x)-P_{0}[h](x)
  &=\int_0^{\infty}\frac{d}{d\theta}P_{\theta}[h](x)d\theta\nonumber\\
  &=\int_0^{\infty}\mathcal{L}\circ D \circ P_{\theta}[h-\langle\gamma,h\rangle](x)d\theta\nonumber\\
  &=\mathcal{L}\circ D[\int_0^{\infty}(P_{\theta}[h]-\langle\gamma,h\rangle)d\theta](x).
\end{align}
Observe that we have added the constant function $\langle\gamma,h\rangle$ to the argument of the operator $\mathcal{L}\circ D,$ which annihilates constants. The purpose of this is  guaranteeing the well-posedness of the infinite integral over $\theta$ appearing at the right hand side of \eqref{interpolationclassic}. From the above discussion it follows that a solution of \eqref{eq:Steinheuristic}, which can be proved to be unique due to elementary results from ODE, is given by the formula 
\begin{align*}
\mathcal{S}[h](x)
  &:=\int_0^{\infty}(P_{\theta}[h](x)-\langle\gamma,h\rangle)d\theta.
\end{align*}
In other words, with the notation previously introduced, $\mathcal{L}\circ D\circ \mathcal{S}[h](x)=h(x)-\langle\gamma,h\rangle$, and consequently,
\begin{align*}
\langle \mu,\mathcal{L}\circ D\circ \mathcal{S}[h]\rangle
  &=\langle \mu,h\rangle-\langle \gamma,h\rangle.
\end{align*}
The regularity properties of $\mathcal{S}[h]$ are typically inherited\footnote{Actually, the solution $\mathcal{S}[h]$ improves the smoothness properties of $h$, rather than inheriting them} from those of $h$ due to fact that  $P_{\theta}[h](x)$ typically improves the regularity properties of $h$ as it is obtained as a convolution against a Gaussian kernel.\\

\noindent \textit{A classical Berry-Esseen type theorem}\\
For expository purposes, assume that $\mathcal{K}$ consists of the elements of $\mathcal{C}^{3}(\R;\R)$, that satisfy $\|f^{\prime\prime\prime}\|_{\infty}\leq 1$. Under this hypothesis, we have that $\|\mathcal{S}[h]^{\prime\prime\prime}\|_{\infty}\leq 1$ for all $h\in\mathcal{K}$. The treatment of the expression $\langle \mu,\mathcal{L}\circ D\circ \mathcal{S}[h]\rangle$ is relatively easy to carry when $\mu$ is the law of a variable $S_{n}$ of the form $S_{n}=\xi_{1,n}+\cdots+\xi_{n,n}$, with $\xi_{j,n}$  independent, square integrable and centered random variables with $\sum_{j} \mathrm{Var}[\xi_{j,n}]=1$. For this particular instance, we can write
\begin{align*}
\langle \mu,\mathcal{L}\circ D\circ \mathcal{S}[h]\rangle
  &=\sum_{j=1}^n\mathbb{E}[\xi_{j,n}D\mathcal{S}[h](S_{n}^{j}+\xi_{j,n})
  -\mathbb{E}[\xi_{j,n}^2]D^{2}\mathcal{S}[h](S_{n})],
\end{align*}
where $S_{n}^{j}:=S_{n}-\xi_{j,n}$. Since $S_{n}^{j}$ and $\xi_{j,n}$ are independent, it then follows that 
$$\mathbb{E}[\xi_{j,n}D\mathcal{S}[h](S_{n}^{j})]=0,$$ 
thus implying that 
\begin{align*}
\langle \mu,\mathcal{L}\circ D\circ \mathcal{S}[h]\rangle
  &=\sum_{j=1}^n\mathbb{E}[\xi_{j,n}
  (D\mathcal{S}[h](S_{n}^{j}+\xi_{j,n})-D\mathcal{S}[h](S_{n}^{j}))
  -\mathbb{E}[\xi_{j,n}^2]D^{2}\mathcal{S}[h](S_{n})].
\end{align*}
By Taylor's theorem, we thus obtain
\begin{align*}
\langle \mu,\mathcal{L}\circ D\circ \mathcal{S}[h]\rangle
  &=\sum_{j=1}^n\mathbb{E}[\xi_{j,n}^3D^{3}\mathcal{S}[h](S_{n}^{j}+\eta_{j,n})],
\end{align*}
for some appropriate random variables $\eta_{j,n}$. From here it follows that 
\begin{align*}
|\langle \mu,\mathcal{L}\circ D\circ \mathcal{S}[h]\rangle|
  &\leq \|D^{3}\mathcal{S}[h]\|_{\infty}\sum_{j=1}^n\mathbb{E}[|\xi_{j,n}^3|],
\end{align*}
thus yielding a sharp bound for $|\langle \mu,\mathcal{L}\circ D\circ \mathcal{S}[h]\rangle|$ for $h\in\mathcal{K}$ and inducing a quantification of the error in the central limit theorem approximation under the metric \eqref{dKdef}.\\

\section{Stein's method in the non-commutative setting}\label{sec:steinmethodlologyNC}	
Now we turn to the main topic of this paper: what happens with the Stein method perspective for proving a central limit theorem when we replace the classical independence of the variables $\xi_{i,n}$ with free independence? \\

\subsection{Outline of the main ideas} A close look to Section \ref{Section:ClassicalSteinMethod} suggest three fundamental parts in the analysis: (i) an analog to  Stein's lemma (ii) a formulation of a Stein equation,  with an adequate analysis of its solution and (iii) An easy implementation to the case of sums of independent random variables. We would like to emphasize the importance of point (iii), as there is a potentially vast choice of different characterizations of the semicircular distribution appearing in Theorem \ref{eq:freclt}.  By examining the proof of the Berry-Esseen theorem in the classical case, one notices that if $\xi_{1,n},\dots, \xi_{n,n}$ are free non-commutative self-adjoint centered random variables defined in $(\mathcal{A},\tau)$ and $S_{n}$ is given by
\begin{align*}
S_{n}
  &:=\sum_{k=1}^n\xi_{k,n},
\end{align*}
the computation of the asymptotics of $\mathbb{E}[S_{n}f(S_{n})]$ is a cornerstone for the solution to the problem in hand. This computation is tractable for the case of free random variables, and as explained in full detail in Section \ref{sec:ncSteineqasd}, the estimation takes the form 
\begin{align*}
\tau[S_{n}f(S_{n})]
  &\approx \tau[(S_{n}-\tilde{S}_{n})^{-1}(f(S_{n})-f(\tilde{S}_{n}))],
\end{align*}
where $\tilde{S}_{n}$ is a tensor independent copy of $S_{n}$. We will also show that in terms of non-commutative differentiable operators, the above identity can be written as follows: define 
 the operator 
 $$
 \begin{array}{ccc}
 \mathcal{L}_{\boxplus}:\mathcal{C}^{1}(\R;\R)&\rightarrow& \mathcal{C}^{0}(\R^2;\R)\\
  g&\mapsto &\mathcal{L}_{\boxplus}[g],
 \end{array}
$$ 
with
\begin{align}\label{eq:Tdef}
\mathcal{L}_{\boxplus}[g](x,y)
  &:=-xg(x)+\frac{g(y)-g(x)}{y-x},
\end{align}
first defined for $x\neq y$ and extended continuously to $\{(x,y)\in\R^{2}\ ;\ x=y\}$. Then we have that 
\begin{align}\label{eq:langleAoverf}
\langle \mu\otimes\mu,\mathcal{L}_{\boxplus}[g]\rangle
  &\approx 0,
\end{align}
for $g=Df$, with $f\in\mathcal{C}^{2}(\R;\R)$. Relation \eqref{eq:langleAoverf} naturally suggests a Stein equation, which can be proved to characterize the semicircular distribution, at the time that offers the possibility to be easily implemented into the framework of the free central limit theorem. In oder to formalize the above free Stein heuristic, the most natural procedure is  to consider the equation
\begin{align}\label{NCsteinone}
\mathcal{L}_{\boxplus}[Df]
  &=h-\langle h,\mathbf{s}\rangle.	
\end{align}
Assuming the well-possednes of the solution $f=\mathcal{S}_{\boxplus}[h]$ to this equation, we would have the identity
\begin{align*}
\mathcal{L}_{\boxplus}\circ D\circ \mathcal{S}_{\boxplus}[h](x,y)
  &= h(x)-\langle \mathbf{s},h\rangle.
\end{align*}
Integrating with respect to $\mu\otimes \mu$ and taking sup over a suitable family of test functions $h$, we should be able to mimic the arguments from Stein's method obtaining an identity of the type 
\begin{align}\label{eq:Steinorig}
\langle \mu\otimes \mu,\mathcal{L}_{\boxplus}\circ D\circ \mathcal{S}_{\boxplus}[h]\rangle
  &= \langle \mu,h\rangle-\langle \mathbf{s},h\rangle.
\end{align}
 This reduces the problem to showing that the left-hand side of \eqref{eq:Steinorig} is approximately zero. While this is a logical approach, the authors have found it challenging to implement, since ensuring the existence of a solution to equation \eqref{NCsteinone} is not evident, let alone proving its regularity properties. In order to avoid dealing with this  problem, we propose a simple alternative approach: instead of aiming for an identity of the type \eqref{eq:Steinorig}, we will follow an interpolation argument close in spirit to the the formulation of the solution to the classical Stein equation by the generator approach, which will yield a natural alternative to the Stein equation \eqref{eq:Steinorig}, in which the solution operator $\mathcal{S}_{\boxplus}$ does not act over the test function $h$, but rather over the underlying measure $\mu$, and   reads 
\begin{align}\label{eq:Steinnc}
\langle \mathcal{S}_{\boxplus}^{*}[\mu],\mathcal{L}_{\boxplus}\circ D[h]\rangle
  &= \langle\mu,h\rangle-\langle \mathbf{s},h\rangle,
\end{align}
for an explicit operator $\mathcal{S}_{\boxplus}^{*}$ defined over measures and taking values over signed measures. Although $\mathcal{S}_{\boxplus}^{*}$ is not defined as an adjoint operator, we have marked it with an upper asterisk  to emphasize the fact that it operates over the left side of the bracket in the dual pairings. Relation \eqref{eq:Steinnc} will be referred to as the dual free Stein equation. The particular shape of $\mathcal{S}_{\boxplus}^{*}$, to be discussed in detail in Section \ref{sec:ncSteineqasd}, formally, takes the form
\begin{align}\label{eq:Udefintegral}
\mathcal{S}_{\boxplus}^{*}[\mu]
  :=\int_0^{\infty}(P_{\theta}^{*}[\mu]\otimes P_{\theta}^{*}[\mu]- \mathbf{s}\otimes \mathbf{s})d\theta,
\end{align}
where $P_{\theta}$ is defined over the set of probability measures $\mathcal{P}(\R)$, takes values on the set $\mathcal{P}(\mathbb{R}^{2})$ of probability measures over $\R^{2}$ and is given as  
\begin{align*}
P_{\theta}^{*}[\mu]
  &:=\dilation_{e^{-\theta}}[\mu]\boxplus \dilation_{\sqrt{1-e^{-2\theta}}}[\mathbf{s}].
\end{align*}

\begin{remark}
As shown in the forthcoming Lemma \ref{Steinslemma}, the identity $\langle \mathbf{s}\otimes \mathbf{s},\mathcal{L}_{\boxplus}[Df]\rangle=0$ holds for continuously differentiable functions with bounded first derivative, yielding 
\begin{align*}
\langle \mathcal{S}_{\boxplus}^{*}[\mu],\mathcal{L}_{\boxplus}[Dh]\rangle
  &=\int_{0}^{\infty}\langle P_{\theta}^{*}[\mu]\otimes P_{\theta}^{*}[\mu],\mathcal{L}_{\boxplus}[Dh]\rangle d\theta,
\end{align*}
so when restricted to test functions of the form $\mathcal{L}_{\boxplus}[Dh]$, the signed measure $\mathcal{S}_{\boxplus}^{*}[\mu]$ can be  thought of as having the formal expression 
\begin{align*}
\mathcal{S}_{\boxplus}^{*}[\mu]
  &=\int_{0}^{\infty}  P_{\theta}^{*}[\mu]\otimes P_{\theta}^{*}[\mu]d\theta.
\end{align*}
The compensator $\mathbf{s}\otimes \mathbf{s}$ in  \eqref{eq:Udefintegral} then serves mainly the purpose of guaranteeing the well-posedness of the action of $\mathcal{S}_{\boxplus}^{*}[\mu]$ over a large domain of functions $g$, instead of only those of the form $g=\mathcal{L}_{\boxplus}[Dh]$. 
\end{remark}
The rest of this section is devoted to developing these ideas. In the sequel, $\mathcal{P}_{\infty}(\R)$ will denote the subset of $\bigcap_{\ell}\mathcal{P}_{\ell}(\R)$, which are characterized by moments. 

\subsection{Non-commutative Stein's Lemma}\label{sec:ncSteineqasd}

In this subsection, we establish an analogue of the Stein lemma for the semicircle distribution. Recall the definition of  $\mathcal{L}_{\boxplus}$, given by \eqref{eq:Tdef}. 


\begin{proposition}[Stein's Lemma]\label{Steinslemma}
If $\mu$ is a  probability measure with moments of arbitrary order, then $\mu$ is the standard semicircular distribution if and only if, for all $f\in \mathcal{C}^{1}(\R;\R)$, 
\begin{equation}
\label{eq:SemicircleSteinLemma}
    \langle \mu\otimes\mu, \mathcal{L}_{\boxplus}[f]\rangle = 0.
\end{equation}
\end{proposition}

\begin{proof}
We first show that if \eqref{eq:SemicircleSteinLemma} holds for every   $f\in \mathcal{C}^{1}(\mathbb{R};\mathbb{R})$, then $\mu=\mathbf{s}$. By taking $f(x) = x^{r}$ with $r\geq0$, we can write
\begin{equation*}
    \langle \mu\otimes\mu , \mathcal{L}_{\boxplus}[f] \rangle = \int_{\R} x^{r+1} \mu(dx) - \int_{\R^2} \frac{y^r-x^r}{y-x} \mu(dx)\mu(dy) = 0.
\end{equation*}
In particular, by taking $r=0$ we obtain that $m_{1}[\mu] = \langle \mu, \iota\rangle = 0$.
Similarly, for all $r\geq 1$,
\begin{equation*}
\int_{\R} x^{r+1} \mu(dx) - \sum_{k=0}^{r-1} \int_{\R^2} y^{k} x^{r-1-k} \mu(dx)\mu(dy)=0.
\end{equation*}
Therefore, the moments of $\mu$ satisfy the relation
\begin{equation*}
    m_{r+1}[\mu] = \sum_{k=0}^{r-1} m_{k}[\mu] m_{r-1-k}[\mu].
\end{equation*}
The above recursive relation and the fact that $m_{1}[\mu]=0$ imply that $m_k[\mu] = m_k[\mathbf{s}]$ for all $k\geq0$. Since the semicircular distribution has compact support, it is characterized by its moments and, as a result, $\mu = s$.\\

\noindent It remains to prove that if $\mu=\mathbf{s}$, then \eqref{eq:SemicircleSteinLemma} holds for all $f\in\mathcal{C}(\R;\R)$. By reversing the argument in the previous paragraph, we can prove that \eqref{eq:SemicircleSteinLemma} holds for polynomials. Then, an approximation argument implies that the same relation holds for all $f\in C^{1}(\mathbb{R};\mathbb{R})$.
\end{proof}


\subsection{Non-commutative Stein's equation}
In this section, we study the Stein equation associated to Lemma \ref{eq:SemicircleSteinLemma}. We begin by introducing the semicircular Ornstein-Uhlenbeck semigroup, which will serve as an interpolation between $\mu$ and $\mathbf{s}$.

\begin{definition}
For each $\theta\geq0$, we define the $\mathcal{P}(\R)$-valued operator $P_{\theta}$, defined over the domain $\mathcal{P}_2(\R)$ by
\begin{equation*}
    P_{\theta}^{*}[\mu] \coloneqq \dilation_{e^{-\theta}}[\mu] \boxplus \dilation_{\sqrt{1-e^{-2\theta}}}[\mathbf{s}].
\end{equation*}
\end{definition}

\noindent It is straightforward to verify that $\{P_{\theta}^{*}\}_{\theta\geq 0}$  is a semigroup over $\mathcal{P}_2(\R)$. It is a semicircular analog of the Gaussian Ornstein-Uhlenbeck semigroup, as
\begin{equation*}
    \langle P_{\theta}^{*}[\mu], h \rangle = \tau\big[h(e^{-\theta}x+\sqrt{1-e^{-2\theta}}z)\big],
\end{equation*}
where $x$ is a selfadjoint non-commutative random variable with analytic distribution $\mu$ and $z$ is a free centered semicircular elements with the same mean and variance as $x$.  We point out that the definition of the dual Ornstein-Uhlenbeck semigroup differs slightly from the earlier standardized version, with this modification being necessary for future computations. Some of our arguments rely on the fact that, as in the classical case,  $\boxplus$ convolution satisfies the following smoothing property, which is a particular case of  \cite[Theorem 1.1]{MR4712710}
\begin{theorem}\label{teo:abscont}
If $\mu,\nu\in\mathcal{P}(\R)$ are such that $\mu$ is absolutely continuous, then $\mu\boxplus\nu	$ is absolutely continuous.
\end{theorem}
From Theorem \ref{teo:abscont}, it follows that $P_{\theta}^{*}[\mu]$ is absolutely continuous for every $\mu\in\mathcal{P}_2(\R)$. This property allows us to prove the following useul representation for the value of $\mathcal{L}_{\boxplus}^{*}$, when restricted to the image of $P_{\theta}^{*}$. 
\begin{proposition}
\label{Theorem:GeneratorSemigroup}
If $\mu\in\mathcal{P}_2(\R)$ and $h\in\mathcal{C}^{1}(\R;\R)$ is bounded, then  
\begin{equation*}
    \lim_{\theta\to0^{+}} \frac{\langle P_{\theta}^{*}[\nu], h \rangle - \langle \mu, h \rangle}{\theta} = \int_{\R^2}\mathcal{L}_{\boxplus}^{*}[Dh](r)(\nu\otimes\nu)(\mathrm{d}r),
\end{equation*}
where $\nu:=P_{\theta}^{*}[\mu].$
\end{proposition}

\begin{proof}
For ease of notation, let $x,y$ be free selfadjoint non-commutative random variables with analytic distributions $\nu$ and $\mathbf{s}$ respectively. For $p\in\mathbb{N}$, let $h_{p}(x) = x^{p}$. By an explicit computation, 
\begin{align*}
\langle P_{\theta}^{*}[\nu], h_{p} \rangle
  &=\tau[(e^{-\theta}x+\sqrt{1-e^{-2\theta}}y)^{p}]\\
  &=\sum_{\epsilon_{1},\dots, \epsilon_{p}\in\{0,1\}}
  e^{-\theta\sum_{l=1}^p\epsilon_l}(1-e^{-2\theta})^{\frac{p-\sum_{l=1}^p\epsilon_l}{2}}\tau[a_{\epsilon_1}\dots a_{\epsilon_{p}}],
\end{align*}
where the $a_{\epsilon_j}$'s are defined by 
$$a_{\epsilon_j}:=\left\{\begin{array}{lll}x& \text{ if } &\epsilon_j=1\\y& \text{ if } &\epsilon_j=0.\end{array}\right.$$
Observe that when $\sum_{l=1}^p\epsilon_l\leq p-3$, we have that 
\begin{align*}
|1-e^{-2\theta}|^{\frac{p-\sum_{l=1}^p\epsilon_l}{2}}
  &\leq C\theta^2,
\end{align*}
for some constant $C>0$ independent of $\theta$.
From here it follows that 
\begin{align*}
|\langle P_{\theta}[\nu], h_{p} \rangle
  -\sum_{\substack{\epsilon_{1},\dots, \epsilon_{p}\in\{0,1\}\\\sum_{l=1}^p\epsilon_l\in\{p,p-1,p-2\}}}
  e^{-\theta\sum_{l=1}^p\epsilon_l}(1-e^{-2\theta})^{\frac{p-\sum_{l=1}^p\epsilon_l}{2}}\tau[a_{\epsilon_1}\dots a_{\epsilon_{p}}]|
  &\leq C\theta^2,
\end{align*}
for a possibly different constant $C$, independent of $\theta$. We thus conclude that 
\begin{align}\label{eq:generaux1}
\lim_{\theta\rightarrow0}\frac{\langle P_{\theta}\nu, h_{p} \rangle - \langle \nu, h_{p} \rangle}{\theta}
    &=\lim_{\theta\rightarrow0}\frac{1}{\theta}\sum_{\substack{\epsilon_{1},\dots, \epsilon_{p}\in\{0,1\}\\\sum_{l=1}^p\epsilon_l=p}}
  (e^{-\theta p}-1)\tau[a_{\epsilon_1}\dots a_{\epsilon_{p}}]\nonumber\\
  &+\lim_{\theta\rightarrow0}\frac{1}{\theta}\sum_{\substack{\epsilon_{1},\dots, \epsilon_{p}\in\{0,1\}\\\sum_{l=1}^p\epsilon_l=p-1}}
  (1-e^{-2\theta})^{\frac{1}{2}}\tau[a_{\epsilon_1}\dots a_{\epsilon_{p}}]\nonumber\\
    &+\lim_{\theta\rightarrow0}\frac{1}{\theta}\sum_{\substack{\epsilon_{1},\dots, \epsilon_{p}\in\{0,1\}\\\sum_{l=1}^p\epsilon_l=p-2}}
  (1-e^{-2\theta})\tau[a_{\epsilon_1}\dots a_{\epsilon_{p}}].
\end{align}
Observe that when $\sum_{l=1}^p\epsilon_l=p-1$, there exists a unique index $j\in\{1,\dots, p\}$ such that $a_{\epsilon_j}=y$, which by the freeness of $x$ and $y$, yields $\tau[a_{\epsilon_1}\dots a_{\epsilon_{p}}]=0$, so the second term in the right hand side of equation \eqref{eq:generaux1} is equal to zero, implying that
\begin{align}\label{eq:freOUaux}
\lim_{\theta\rightarrow0}\frac{\langle P_{\theta}\nu, h_{p} \rangle - \langle \nu, h_{p} \rangle}{\theta}
    &=-p\tau[x^p]
    +2\sum_{\substack{\epsilon_{1},\dots, \epsilon_{p}\in\{0,1\}\\\sum_{l=1}^p\epsilon_l=p-2}}\tau[a_{\epsilon_1}\dots a_{\epsilon_{p}}].
\end{align}
By localizing the unique two indices $i,j$ satisfying $a_{\epsilon_i}=a_{\epsilon_j}=s$, we observe that the indices $\epsilon_1,\dots, \epsilon_p$ satisfying $\epsilon_1+\dots+\epsilon_p$ are in bijection with the partitions $\pi$ of size two over $\{1,\dots, p\}$. We denote by $\pi[x,y]$ the value of $\tau[a_{\epsilon_1}\dots, a_{\epsilon_p}]$ associated to the corresponding bijection 
$(\epsilon_1,\dots,\epsilon_p)\mapsto \pi$. Observe that $\pi[x,y]
  =\tau[x^{a}yx^byx^c]$ for some $a,b,c\geq 0$. Using the trace property of $\tau$, we can rewrite this expression in the form 
  $\pi[x,y]=\tau[yx^byx^{a+c}]=\tau[yx^{a+c}yx^b]$. Applying \eqref{eq:firstmomentsfreeness}, we thus conclude that 
\begin{align}\label{eq:Pixyeq}
\pi[x,y]
  &=\tau[y^2]\tau[x^{\ell_{\pi}}]\tau[x^{q-\ell_{\pi}-2}],
\end{align}
where $\ell_{\pi}$ is such that $0\leq \ell\leq q-\ell-2$ and denotes the exponent $b$  when $b\leq a+b$ or the exponent $a+b$ otherwise. Combining \eqref{eq:freOUaux} and \eqref{eq:Pixyeq}, and using the fact that $y$ is standardized, we conclude that 
\begin{align*}
\lim_{\theta\rightarrow0}\frac{\langle P_{\theta}\nu, h_{p} \rangle - \langle \nu, h_{p} \rangle}{\theta}
    &=-p\tau[x^p]
    +2\sum_{l=0}^{\lfloor q/2-1\rfloor}\sum_{\pi}\Indi{\{\ell_{\pi}=l\}}\tau[x^{l}]\tau[x^{q-l-2}],
\end{align*}
where the sum ranges over the partitions $\pi$ described above. One can easily check that the number of $\pi$ satisfying $\ell_{\pi}=l$ is equal to $(p/2)\Indi{\{\ell=p/2\}}+p\Indi{\{\ell\neq p/2\}}$, which yields
\begin{align*}
\lim_{\theta\rightarrow0}\frac{\langle P_{\theta}\nu, h_{p} \rangle - \langle \nu, h_{p} \rangle}{\theta}
    &=-p\tau[x^p]
    +p\sum_{l=0}^{q-2}\tau[x^{l}]\tau[x^{q-l-2}],
\end{align*}
By the absolute continuity of $\nu$, the diagonal of $\R^{2}$ is $(\nu\otimes\nu)$-null, and the identity 
\begin{align*}
p\sum_{l=0}^{q-2}\tau[x^{l}]\tau[x^{q-l-2}]
  &=\int_{\R^2}\frac{pu^{p-1}-pv^{p-1}}{u-v}\nu(du)\nu(dv)
\end{align*}
holds. From here it follows that  
\begin{align*}
\lim_{\theta\rightarrow0}\frac{\langle P_{\theta}\nu, h_{p} \rangle - \langle \nu, h_{p} \rangle}{\theta}
    &=-p\int_{\R}u^p\nu(du)+\int_{\R^{2}}\frac{pu^{p-1}-pv^{p-1}}{u-v}\nu(du)\nu(dv)\\
    &=-\int_{\R}h_p(u)\nu(du)+\int_{\R^{2}}\frac{Dh_p(u)-Dh_p(u)}{u-v}\nu(du)\nu(dv).
\end{align*}
This finishes the proof of the result for the case $h(x)=x^p$. The result for general $h$ follows by an approximation argument, which holds due to the fact that $\nu\otimes \nu$ does not charge mass over the diagonal. 
\end{proof}


\subsection{The solution to the dual free Stein equation}
Next we describe the solution $\mathcal{S}_{\boxplus}^{*}$ to the Stein equation. We begin with a preliminary technical result.
\begin{lemma}
Let $\mu$ be a probability measure with finite second moment. For a given $f\in\mathcal{C}(\R^2;\R)$ Lipchitz, the integral 
$$\int_0^{\infty}|\langle P_{\theta}[\mu]\otimes P_{\theta}[\mu], f\rangle-\langle s, f\rangle|d\theta$$ 
is finite. In addition, the mapping  
\begin{align}\label{eq:Udef}
f\mapsto\int_{0}^{\infty} \langle P_{\theta}[\mu_{n}]\otimes P_{\theta}[\mu_{n}]-  \mathbf{s}\otimes \mathbf{s},f\rangle  d\theta,
\end{align}
defined over the set of smooth Lipchitz functions, 
induces a signed measure. 
\end{lemma}

\begin{proof}
Let $f:\R^{2}\rightarrow\R$ be a continuously differentiable Lipchitz function. For $u,v\in\R$, we define $f_{1,u},f_{2,v}:\R\rightarrow\R$ by $f_{1,u}(y)=f(u,y)$ and $f_{2,v}(x):=f(x,v)$. Then, if $X$ and $Y$ are non-commutative free random variables defined in $(\mathcal{A},\tau)$, with probability distributions $\mu$ and $\mathbf{s}$ respectively, then 
\begin{align*}
|\langle P_{\theta}[\mu_{n}]\otimes P_{\theta}[\mu_{n}]-  \mathbf{s}\otimes \mathbf{s},f\rangle|
  &\leq |\langle P_{\theta}[\mu_{n}]\otimes \mathbf{s}-  \mathbf{s}\otimes \mathbf{s},f\rangle|+|\langle P_{\theta}[\mu_{n}]\otimes P_{\theta}[\mu_{n}]-  P_{\theta}[\mu_{n}]\otimes \mathbf{s},f\rangle|\\
  &\leq 2\sup_{x,y\in\R}|\langle P_{\theta}[\mu_{n}]-   \mathbf{s},f_{2,y}(x)\rangle|\vee|\langle P_{\theta}[\mu_{n}]-   \mathbf{s},f_{2,y}(x)\rangle|.
\end{align*}
We can easily check that $f_{i,r}$ is Lipchitz, and consequently, 
\begin{align*}
|\langle P_{\theta}[\mu_{n}]\otimes P_{\theta}[\mu_{n}]-  \mathbf{s}\otimes \mathbf{s},f\rangle|
  &\leq 2\sup_{\substack{g\in\mathcal{C}^1(\R;\R)\\|Dg|\leq 1}}|\langle P_{\theta}[\mu_{n}]- \mathbf{s},g(x)\rangle|.
\end{align*}
Let $g:\R\rightarrow\R$ be a continuously differentiable Lipchitz function. Then,  
\begin{align*}
\langle P_{\theta}[\mu_{n}]- \mathbf{s},g(x)\rangle|
  &=|\tau[g(e^{-\theta}X+\sqrt{1-e^{-2\theta}}Y)-g(Y)]|\\
  &\leq \|Dg\|_{\infty}
  \tau[e^{-\theta}|X|+|\sqrt{1-e^{-2\theta}}-1||Y|].
\end{align*}
In particular, if $\theta\geq 1$, the integrability of $Y$ yields
\begin{align}\label{eq:OUestimate}
|\langle P_{\theta}[\mu_{n}]\otimes P_{\theta}[\mu_{n}]-  \mathbf{s}\otimes \mathbf{s},f\rangle|
  &\leq 6\|Df\|_{\infty}e^{-\theta}\leq 6e^{-\theta}.
\end{align}
The fact that $\mathcal{S}_{\boxplus}^{*}[\mu]$ is a well-defined probability measure  follows by approximating the infinite integral by a large compact set. The argument can be  formalized by using  \eqref{eq:OUestimate}.
\end{proof}
The theorem bellow provides a solution to the dual free Stein equation.
\begin{proposition}\label{prop:main}
If $\mu$ is a probability measure with moments of arbitrary order, then the equation 
\begin{align*}
\langle\mathbf{s},h\rangle 
- \langle\mu,h\rangle
    &=  \langle \nu,\mathcal{L}_{\boxplus}[Dh]\rangle,
\end{align*}
admits $\nu=\mathcal{S}_{\boxplus}^{*}[\mu]$ as a solution.
\end{proposition}
\begin{proof}
We first write
\begin{align*}
\langle\mathbf{s},h\rangle 
- \langle\mu,h\rangle
    &=  \big(\langle P_{\infty}\mu_{n}, h \rangle - \langle P_{0}\mu_{n}, h \rangle\big)\\
    &=  \int_{0}^{\infty} \frac{\mathrm{d}}{\mathrm{d}\theta} \langle P_{\theta}[\mu_{n}],h\rangle \mathrm{d}\theta.
\end{align*}
One can easily check that $\mathbf{s}$ is a fixed point for $P_{\theta}$, and consequently, by the above identity,
\begin{align*}
\langle\mathbf{s},h\rangle 
- \langle\mu,h\rangle
    &=  \int_{0}^{\infty} \frac{\mathrm{d}}{\mathrm{d}\theta} (\langle P_{\theta}[\mu_{n}],h\rangle-\langle \mathbf{s},h\rangle) \mathrm{d}\theta\\
    &=  \int_{0}^{\infty} \langle P_{\theta}[\mu_{n}]\otimes P_{\theta}[\mu_{n}]-  \mathbf{s}\otimes \mathbf{s},\mathcal{L}_{\boxplus}[Dh]\rangle \mathrm{d}\theta.
\end{align*}
\end{proof}


\subsection{Notational prelude to the free Berry-Esseen theorem}
In this section we introduce notation that will allow us to substantially simplify our technical computations. In the sequel, $\mathfrak{s}$ will denote the symmetrization operator, acting over non-commutative variables. That is to say, for every polynomial $f\in\mathbb{C}[X_1,\dots, X_m]$, we define the polynomial
\begin{align*}
\mathfrak{s}[f](X_1,\dots, X_{m})
  &:=\sum_{\sigma\in\mathfrak{S}_n}	\frac{1}{m!}f(X_{\sigma_1},\dots, X_{\sigma_m}),
\end{align*}
where $\mathfrak{S}_n$ denotes the set of permutations over $n$ elements. For $z$ belonging to the upper half plane, we define as well the polynomial
\begin{align*}
\Delta\left(a,r\right)
  &:=2\mathfrak{s}[(z-a)r]-r^2.
\end{align*}
Consider the set of multi-indices $\mathcal{I}:=\bigcup_{l=1}^{\infty}\N^{2l}$. For all $j\geq 1$, there exist universal constants $\{\mathfrak{f}_{j,\mathbbm{i}}^1,\mathfrak{f}_{j,\mathbbm{i}}^2\ ;\ \mathbbm{i}\in \mathcal{I}\}$ such that  $\mathfrak{f}_{j,\mathbbm{i}}^1=\mathfrak{f}_{j,\mathbbm{i}}^2=0$ for $|\mathbbm{i}|\geq 2j$, and a collection of polynomials $\{\mathfrak{Q}_{j,\mathbbm{i}}^1,\mathfrak{Q}_{j,\mathbbm{i}}^2\ ;\ \mathbbm{i}\in \mathcal{I}\}$ only depending on $j$, such that 
\begin{align}\label{eq:fullexpansion}
(a\Delta\left(a,r\right))^j
  &=\sum_{\mathbbm{i}\in\mathcal{I}}(\mathfrak{f}_{j,\mathbbm{i}}^1	\mathfrak{Q}_{j,\mathbbm{i}}^1(z)\Upsilon_{1,\mathbbm{i}}(a,r)
  +\mathfrak{f}_{j,\mathbbm{i}}^2	\mathfrak{Q}_{j,\mathbbm{i}}^2(z)\Upsilon_{2,\mathbbm{i}}(a,r)),
\end{align}
where 
\begin{align}\label{eq:Upsilondef}
\Upsilon_{1,\mathbbm{i}}(a,r)
  &:=a^{i_1}r^{i_2}\cdots a^{i_{|\mathbbm{i}|}}r^{i_{|\mathbbm{i}|}}	\nonumber\\
  \Upsilon_{2,\mathbbm{i}}(a,r)
  &:=r^{i_1}a^{i_2}\cdots r^{i_{|\mathbbm{i}|}}a^{i_{|\mathbbm{i}|}}	,
\end{align}
with $i_{\ell}\geq 1$ for all $\ell\geq 1$.

\section{Berry-Esseen type theorems for weakly dependent variables}\label{freeberryessensec}
This section is devoted to assessing the rate of convergence for sums of weakly dependent non-commutative random variables, with improvement in the rate under suitable moment conditions.  We begin introducing the notion of dependency graph.\\

\noindent\textit{Dependency graphs}\\
This brief subsection is mainly taken from  \cite{MR1986198}, reference to which the reader is referred for a detailed discussion on applications to the study of Poisson and Gaussian applications in random graphs. 
\begin{definition}
Suppose $([n],E)$ is a graph over $[n]$. For $i,j\in[n]$, we write $i\sim j$ if $\{i,j\}\in E$. For $i\in[n]$, we let $N_i$ denote the adjacency neighbourhood of $i$, defined as the set 
$$N_i:=\{i\}\cup\{j\in[n]; j\sim i\}.$$ 
We say that the graph $([n],E)$ is a dependency graph for a collection of random variables $\{\xi_i\ ;\ i\in[n]\}$ if for any two disjoint subsets $I_1, I_2$ of $[n]$ such that there are no edges connecting $I_1$ to $I_2$, the collection of random variables $\{\xi_i\ ;\  i \in I_1 \}$ is free from $\{\xi_i\ ;\  i \in I_2 \}$.
\end{definition}
\noindent We observe that in the case where the graph  has no edges, we recover the classical notion of freeness.\\

\noindent\textit{Moment matching rank}\\
In the particular case where $\xi_{k,n}$ forms a semicircular family of variables, it is clear that the accuracy of the semicircular approximation in the free central limit theorem is perfect. Moreover, according to the findings of Salazar in \cite{MR4607696}, in the homogeneous, fully free case, a third moment vanishing condition for the $\xi_{k,n}$ implies that the rate of accuracy in the free CLT can be quadratically improved.  It becomes natural, then, to wonder whether this phenomenon has an analog with a higher order of improvement in the rate. To establish the appropriate framework for addressing this question, we consider a collection $\{s_{k,n}\}_{k \geq 1}$ of jointly semicircular random variables with first and second moments identical to those of $\{\xi_{k,n}\}_{k \geq 1}$.\\

\noindent  In the sequel, for a given element $\rho\in\mathcal{P}_2(\R)$, we will denote by $\mathfrak{g}[\rho]$ the semicircular distribution with the same moments of order one and two of $\rho$.

\begin{definition}
For a given $\rho\in\mathcal{P}_2(\R)$, the moment matching rank $q[\rho]$ associated to $\rho$ is the maximum of the set of natural numbers $\ell$ satisfying $m_{j}[\mu_{k,n}]=m_{j}[\mathfrak{g}[\mu_{k,n}]]$. The matching rank of a sequence of probability measures ${\rho}=\{\rho_{i}\ ;\ i\in I\}$, with $I$ being an arbitrary index is defined as the maximum rank in the components of ${\rho}.$
\end{definition}
The following lemma is a direct consequence of the definition of moment matching rank

\begin{lemma}\label{lem:momentmatchingPUrel}
If a given measure $\rho$ has moment matching rank $q$, then for all $j\in[q]$ and $\theta\geq 0$, 
	\begin{align*}
	m_j[\rho]
	  &=m_j[\mathfrak{g}[\rho]]
	  =m_j[P_{\theta}[\mathfrak{g}[\rho]]]=m_j[P_{\theta}[\rho]].	
	\end{align*}
\end{lemma}

\noindent\textit{Main applications}\\
Our main applications are presented next. For its statement, we will introduce the following notation. For a given collection of non-commutative random variables $\xi^{n}=\{\xi_{k,n}\}_{k\geq 1}$, we will choose a dependency graph $E=E_{n}$. The set of equivalence classes of $E$ will be denoted by $\mathcal{J}=\mathcal{K}_{E_n}$. For a given element $V\in\mathcal{J}$, we will denote by $\xi_{V}$ the random variable
\begin{align*}
\xi_{V}
  &:=\sum_{i\in V}\xi_{i,n}.	
\end{align*}
The distribution of $\xi_V$ will be denoted by $\mu_V$.
The next theorem provides an assessment of the distance towards semicircularity
\begin{theorem}\label{thm:inhomogeneousBerryEsseen}
 Assume that the sequence $\xi^n$ has dependency graph $E_n$ and denote by $\mathcal{D}(E_n)$ the maximum degree of the $E_n$. Let $\mu_{k,n}$ and $\nu_n$ denote the analytic distributions of $\xi_{k,n}$ and $\xi_{1,1}+\cdots+\xi_{n,n}$, respectively. Then,  for $n$ sufficiently large,  
\begin{align*}
d_{TV}(\nu_n,\mathbf{s})
  &\leq C\mathcal{D}(E_n)^2\sum_{k=1}^nm_{3}[\mu_{k,n}],
\end{align*}
for some constant $C>0$ independent of $n$. 
\end{theorem}
If the distance $d_{TV}$ is replaced by the Kantorovich-Rubenstein-Wasserstein distance, we obtain the following refinement of the above theorem
\begin{theorem}\label{thm:inhomogeneousBerryEsseensecond}
 Assume that the sequence $\xi^n$ has dependency graph $E_n$. Then,  under the condition  
\begin{align*}
\lim_{n}\mathcal{D}(E_n)^2\sum_{k=1}^{n}m_3[\mu_{k,n}]=0,
\end{align*} 
we have that for $n$ sufficiently large,  
\begin{align*}
d_{W}(\nu_n,\mathbf{s})
  &\leq C\mathcal{D}(E_n)^{q+1}\sum_{k=1}^nm_{q+1}[\mu_{k,n}],
\end{align*}
for some constant $C>0$ independent of $n$. 
\end{theorem}

\section{Proofs of main results}\label{sec:proofofmainresults}
This section is devoted to the proof of Theorems \ref{thm:inhomogeneousBerryEsseen} and \ref{thm:inhomogeneousBerryEsseensecond}.
\subsection{Proof of Theorem \ref{thm:inhomogeneousBerryEsseen}}
By Proposition \ref{prop:main}, finding bounds for 
$\langle \mathbf{s},h\rangle-\langle \nu_n,h\rangle $
simplifies to bounding from above the term
\begin{align*}
\langle \mathcal{S}_{\boxplus}^{*}[\nu_n],\mathcal{L}_{\boxplus}[Dh]\rangle.
\end{align*}
We will first consider the case where $h\in \mathcal{C}^1(\R;\R)$ and $\|f\|_{\infty}\leq 1$. By \eqref{eq:Udefintegral},
\begin{align*}
\langle \mathcal{S}_{\boxplus}^{*}[\nu_n],\mathcal{L}_{\boxplus}[Dh]\rangle
  =\int_0^{\infty}\langle P_{\theta}^{*}[\nu_n]\otimes P_{\theta}^{*}[\nu_n]- \mathbf{s}\otimes \mathbf{s},\mathcal{L}_{\boxplus}[Dh]\rangle d\theta.
\end{align*}

The term $\langle \mathbf{s}\otimes \mathbf{s},\mathcal{L}_{\boxplus}[Dh]\rangle$ is equal to zero by Proposition \ref{Steinslemma}, so we can write 
\begin{multline*}
\langle P_{\theta}^{*}[\nu_{n}]\otimes P_{\theta}^{*}[\nu_{n}]- \mathbf{s}\otimes \mathbf{s},\mathcal{L}_{\boxplus}[Dh]\rangle\\
\begin{aligned}
  &=\int_{\R^{2}}\left(-xDh(x)+\frac{Dh(x)-Dh(y)}{x-y}\right)P_{\theta}[\nu_{n}](dx)P_{\theta}[\nu_{n}](dy).
\end{aligned}
\end{multline*}
The above identity, combined with Proposition \ref{prop:main} allows us to write 
\begin{align}\label{eq:auxmainthmone}
\langle \mathbf{s},h\rangle-\langle \nu_n,h\rangle 
  &=	\int_0^{\infty}\int_{\R^{2}}\left(-xDh(x)+\frac{Dh(x)-Dh(y)}{x-y}\right)P_{\theta}[\mu_{n}](dx)P_{\theta}[\mu_{n}](dy)d\theta. 
\end{align}
The rest of the proof will consist of finding sharp bounds for the right-hand side by making use of the Cauchy formula.\\

\noindent Recall that $\mathcal{J}$ denotes the set of equivalence classes of the dependency graph $E$. Denote by $\tilde{\mu}_V$, with $V\in \mathcal{J}$, the standardization of the analytic distribution of $\xi_{V}$. By the freeness of elements in different equivalence classes of $E$, we have that $\nu_n$ can be expressed as the free convolution of the $\tilde{\mu}_V$'s, with $V$ ranging over $\mathcal{J}$. By the triangle inequality, the support of the $\tilde{\mu}_V$ is contained in $\mathcal{D}R$, where $R$ is any positive number satisfying $\supp(\mu_k)\subset[-R,R]$. One can easily check that the validity of \cite[Theorem 7]{MR1355057} can be extended to the measure $\nu_n$,  so by the superconcentration criterion \cite[Theorem 3]{MR1355057}, there exists $N\geq 1$, such that the measure $\nu_n$ is supported in $[-3,3]$, yielding the condition $\supp P_{\theta}[\nu_{n}]\subset[-5,5]$. From here it follows that if $T:=\Indi{[-5,5]}$, then the push-forward measure $T_{\#}P_{\theta}[\nu_{n}]$ coincides with $P_{\theta}[\nu_{n}]$, which allows us to write 
\begin{multline*}
\langle P_{\theta}[\nu_{n}]\otimes P_{\theta}[\nu_{n}]- \mathbf{s}\otimes \mathbf{s},\mathcal{L}_{\boxplus}[Dh]\rangle\\
\begin{aligned}
  &=\int_{[-5,5]^{2}}\left(-xDh(x)+\frac{Dh(x)-Dh(y)}{x-y}\right) P_{\theta}[\nu_{n}](dx) P_{\theta}[\nu_{n}](dy).
\end{aligned}
\end{multline*}
Observe that for all $z\in[-5,5]$, the Cauchy integral formula yields 
\begin{align*}
h(x)
  =\frac{1}{2\pi\mathbf{i}}\int_{\mathcal{R}}\frac{h(z)}{z-x}dz, \ \ \ \ \ \ \ 
Dh(x)
  =\frac{1}{2\pi\mathbf{i}}\int_{\mathcal{R}}\frac{h(z)}{(z-x)^{2}}dz,	
\end{align*}
where $\mathcal{R}$ is the rectangle determined by the vertices $(\pm 6,\pm 1).$ From here it follows that 
\begin{multline*}
\langle P_{\theta}[\nu_{n}]\otimes P_{\theta}[\nu_{n}]- \mathbf{s}\otimes \mathbf{s},\mathcal{L}_{\boxplus}[Dh]\rangle\\
\begin{aligned}
    &=\frac{1}{2\pi\mathbf{i}}\int_{\mathcal{R}}h(z)\int_{[-5,5]^{2}}\left(-\frac{x}{(z-x)^{2}}-\frac{(x-z)+(y-z)}{(z-x)^{2}(z-y)^{2}}\right) P_{\theta}[\nu_{n}](dx) P_{\theta}[\nu_{n}](dy)dz.
\end{aligned}
\end{multline*}
Thus, by a symmetrization argument, 
\begin{multline*}
\langle P_{\theta}[\nu_{n}]\otimes P_{\theta}[\nu_{n}]- \mathbf{s}\otimes \mathbf{s},\mathcal{L}_{\boxplus}[Dh]\rangle\\
\begin{aligned}
    &=\frac{1}{2\pi\mathbf{i}}\int_{\mathcal{R}}h(z)\int_{[-5,5]^{2}}\left(-\frac{x}{(z-x)^{2}}-\frac{2(x-z)}{(z-x)^{2}(z-y)^{2}}\right) P_{\theta}[\nu_{n}](dx) P_{\theta}[\nu_{n}](dy)dz.
\end{aligned}
\end{multline*}
Using the fact that $\|h\|_{\infty}\leq 1$ and that the perimeter of $\mathcal{R}$ is 28, and then applying the bound \eqref{eq:auxmainthmone}, we deduce the bound 
\begin{multline}\label{eq:technicalone}
|\langle \mathbf{s},h\rangle-\langle \nu_n,h\rangle|=\langle \mathcal{S}_{\boxplus}^{*}[\nu_n],\mathcal{L}_{\boxplus}[Dh]\rangle\\
\begin{aligned}
	&=|\int_{\R_{+}}\langle P_{\theta}[\nu_{n}]\otimes P_{\theta}[\nu_{n}]- \mathbf{s}\otimes \mathbf{s},\mathcal{L}_{\boxplus}[Dh]\rangle d\theta|\\
    &\leq \frac{14}{\pi}\int_{\R_{+}}\sup_{z\in\mathcal{R}}\left|\int_{[-5,5]^{2}}\left(\frac{x}{(z-x)^2}+\frac{2(x-z)}{(z-x)^{2}(z-y)^{2}}\right) P_{\theta}[\nu_{n}](dx) P_{\theta}[\nu_{n}](dy)\right|d\theta.
\end{aligned}
\end{multline}
To handle the argument in the supremum appearing in the right, define the function 
$$g(x):=(z-x)^{-2}.$$ 
Relation \eqref{eq:technicalone} then  reads
\begin{multline*}
|\langle \mathbf{s},h\rangle-\langle \nu_n,h\rangle|\\
\begin{aligned}
    &\leq \frac{14}{\pi}\int_{\R_{+}}\sup_{z\in\mathcal{R}}\left|\int_{[-5,5]^{2}}\left(xg(x)+2(x-z)g(x)g(y)\right) P_{\theta}[\nu_{n}](dx) P_{\theta}[\nu_{n}](dy)\right|d\theta.
\end{aligned}
\end{multline*}
This expression allows us to transfer ideas from classical Stein method discussed in Section \ref{Section:ClassicalSteinMethod}, for transforming the expression 
\begin{align*}
\int_{[-5,5]}xg(x)P_{\theta}[\nu_{n}](dx),	
\end{align*}
into the negative of
\begin{align*}
\int_{[-5,5]^{2}}2(x-z)g(x)g(y) P_{\theta}[\nu_{n}](dx) P_{\theta}[\nu_{n}](dy),
\end{align*}
plus a quantifiable error. In order to transparently carry out this program, we first introduce some notation. Let $\{s_{V},\tilde{s}_{V}\ ;\ V\in\mathcal{J}\}$ be a sequence of free random variables with $s_{V}, \tilde{s}_V$ distributed according to $\mathfrak{g}[\mu_{\xi_{V}}]$. Define as well the variables 
\begin{align*}
\eta_{V}^{\theta}
  &:=e^{-\theta}\xi_{V}+\sqrt{1-e^{-2\theta}}s_{V},	
\end{align*}
as well as  
\begin{align*}
F_{\theta,n}
	:=\sum_{V\in\mathcal{J}}\eta_{V}^{\theta}
\ \ \ \ \ \ \ \
F_{\theta,n}^{V}
	:=\tilde{\eta}_{V}^{\theta}+\sum_{U\in\mathcal{J}\backslash\{V\}} \eta_{U}^{\theta},
\end{align*}
where the $\tilde{\eta}_{V}^{\theta}$ are free copies of the ${\eta}_{V}^{\theta}$. From the definition of these variables, the properties bellow follow directly
\begin{enumerate}
\item[-]	The variables  $\eta_{V}^{\theta} $ are free, and free from the $\tilde{\eta}_{V}^{\theta}$, as $V$ ranges over $\mathcal{J}$.
\item[-]	The analytic distribution of $\eta_{V}^{\theta}$ is  
$P_{\theta}[\mu_{V}]$, where $\mu_V$ denotes the analytic distribution of $\xi_{V}$.
\item[-]	The analytic distributions of $F_{\theta,n}$ and $F_{\theta,n}^V$ are equal to $P_{\theta}[\nu_n]$ for all $V\in\mathcal{J}$.
\end{enumerate}
Inspired in the procedure described in Section \ref{sec:steinmethodlologyNC}, we start from the expression
\begin{align*}
\int_{[-5,5]}xg(x)P_{\theta}[\nu_{n}](dx)
  &=\tau\left[\sum_{V\in\mathcal{J}}\eta_{V}^{\theta}g\left(F_{\theta,n}\right)\right].
\end{align*}
By the freeness of the $\eta_{V}^{\theta}$ against $\eta_{U}^{\theta}$, for $V\neq U$, we deduce that for all $V\in\mathcal{J}$,
$$\tau\left[\eta_{V}^{\theta}g\left(F_{\theta,n}^{V}\right)\right]=0,$$  
which leads to 
\begin{align*}
\int_{[-5,5]}xg(x)P_{\theta}[\nu_{n}](dx)
  &=\tau\left[\sum_{V\in\mathcal{J}}\eta_{V}^{\theta}\left(g\left(F_{\theta,n}\right)-g\left(F_{\theta,n}^{V}\right)\right)\right].
\end{align*}
We now approach the right-hand side using the framework of non-commutative differentiable calculus. It is worth noting that a significant simplification has been achieved by applying the Cauchy inversion formula. This transformation reduced the complexity of the expressions to be managed, as the differential calculus for the function $g$  is considerably simpler compared against the original function  $h$. In the sequel, $\zeta_{V}^{\theta}$ will denote the difference 
\begin{align*}
\zeta_{V}^{\theta}
  &:=\eta_{V}^{\theta}-\tilde{\eta}_V^{\theta}.	
\end{align*}
Denote by $\mathfrak{s}$ the symmetrization operator over polynomials in non-commutative variables and define the functions 
\begin{align*}
\Delta\left(a,r\right)
  &:=2\mathfrak{s}[(z-a)r]-r^2,
\end{align*}
for non-commutative variables $a,r\in\mathcal{A}$. By  Lemma \ref{lemma:atechone}, we have that
\begin{align}\label{eq:decompmainone}
\int_{[-5,5]}xg(x)P_{\theta}[\nu_{n}](dx)
  &=\sum_{V\in\mathcal{J}}\tau\left[\eta_{V}^{\theta}g(F_{\theta,n}^V) {\Delta}\left(F_{\theta,n}^V,\zeta_{V}^{\theta}\right)g(F_{\theta,n}^V))\right]+\mathcal{E}_{\theta,n},
\end{align}
where
\begin{align*}
\mathcal{E}_{\theta,n}(z)
  &:=\sum_{V\in\mathcal{J}}\tau\left[\eta_{V}^{\theta}g(F_{\theta,n})({\Delta}\left(F_{\theta,n}^V,\zeta_{V}^{\theta}\right)g(F_{\theta,n}^V))^2\right].
\end{align*}
Similarly, we can consider the decomposition
\begin{align*}
\int_{\R^{2}}\frac{2(x-z)}{(z-x)^{2}(z-y)^{2}}P_{\theta}[\nu_{n}](dx) P_{\theta}[\nu_{n}](dy)	
  &=2\tau[(F_{\theta,n}-z)g(F_{\theta,n})]\tau[g(F_{\theta,n})]\\
  &=\sum_{V\in\mathcal{J}}\tau[\eta_V^{\theta}\zeta_{V}^{\theta}]2\tau[(F_{\theta,n}-z)g(F_{\theta,n})]\tau[g(F_{\theta,n})].
\end{align*}
Now, using the fact that $F_{\theta,n}$ and $F_{\theta,n}^V$ are equal in law, we deduce the identity 
\begin{align*}
\int_{\R^{2}}\frac{2(x-z)}{(z-x)^{2}(z-y)^{2}}P_{\theta}[\nu_{n}](dx) P_{\theta}[\nu_{n}](dy)	
  &=\sum_{V\in\mathcal{J}}\tau[\eta_V^{\theta}\zeta_{V}^{\theta}]2\tau[(F_{\theta,n}^V-z)g(F_{\theta,n}^V)]\tau[g(F_{\theta,n}^V)].
\end{align*}
Moreover, by \eqref{eq:firstmomentsfreeness}, the right-hand side coincides with 
\begin{align*}
2\sum_{V\in\mathcal{J}}\tau\left[\eta_{V}^{\theta}g(F_{\theta,n}^V) \mathfrak{s}[(z-F_{\theta,n}^V){\zeta}_{V}^{\theta}]g(F_{\theta,n}^V)\right].
\end{align*}
From here it follows that the term 
\begin{align*}
\mathcal{Y}_{\theta,z}
  &:=\left|\int_{[-5,5]^{2}}\left(xg(x)+2(x-z)g(x)g(y)\right) P_{\theta}[\nu_{n}](dx) P_{\theta}[\nu_{n}](dy)\right|,
\end{align*}
satisfies 
\begin{align*}
\mathcal{Y}_{\theta,z}
	&=|-\sum_{V\in\mathcal{J}}\tau\left[\eta_{V}^{\theta}g(F_{\theta,n}^V)(\zeta_{V}^{\theta})^2 F_{\theta,n}^V\right]+\mathcal{E}_{\theta,n}|.
\end{align*}
By the non-commutative Stein lemma, $\mathcal{Y}_{\infty,z}=0$, so that 
\begin{align*}
\mathcal{Y}_{\theta,z}
  &=|\mathcal{Y}_{\theta,z}-\mathcal{Y}_{\infty,z}|\\
  &\leq \sum_{V\in\mathcal{J}}|\tau\left[\eta_{V}^{\theta}g(F_{\theta,n}^V)(\zeta_{V}^{\theta})^2 F_{\theta,n}^V\right]-\tau\left[\eta_{V}^{\infty}g(F_{\infty,n}^V)(\zeta_{V}^{\infty})^2 F_{\infty,n}^V\right]|\\
  &+\sum_{V\in\mathcal{J}}|\tau\left[\eta_{V}^{\theta}g(F_{\theta,n})({\Delta}\left(F_{\theta,n}^V,\zeta_{V}^{\theta}\right)F_{\theta,n}^V)^2\right]
  -\tau\left[\eta_{V}^{\infty}g(F_{\infty,n})({\Delta}\left(F_{\infty,n}^V,\zeta_{V}^{\infty}\right)F_{\infty,n}^V)^2\right]|.
\end{align*}
The boundedness of $g$, its derivatives and the support of $F_{\theta,n}^{V}$, 
\begin{align*}
\mathcal{Y}_{\theta,z}
  &\leq Ce^{-\theta}\sum_{V\in\mathcal{J}} \tau\left[|\eta_{V}^{\theta}|^3+|\zeta_V^{\theta}|^3\right].
  \end{align*}
From here we can easily deduce the existence of a (possibly different) constant $C>0$, such that 
\begin{align*}
\mathcal{Y}_{\theta,z}
  &\leq Ce^{-\theta}\sum_{V\in\mathcal{J}} \tau\left[|\xi_{V}^{\theta}|^3\right].
  \end{align*}
An application of the H\"older inequality then yields 
\begin{align*}
\mathcal{Y}_{\theta,z}
  &\leq C\mathcal{D}(E_n)^2e^{-\theta}\sum_{k=1}^n \tau\left[|\xi_{k,n}^{\theta}|^3\right].
  \end{align*}
Thus, by \eqref{eq:technicalone}, 
\begin{align*}
|\langle \mathbf{s},h\rangle-\langle \nu_n,h\rangle|
  &\leq C\mathcal{D}(E_n)^2\sum_{k=1}^n m_3[\mu_{k,n}].
\end{align*}
The result is obtained by taking supremum over $h$ differentiable, bounded by one.

\subsection{Proof of Theorem \ref{thm:inhomogeneousBerryEsseensecond}}\label{ref:inhomogeneousBerryEsseen23}
Observe that the measure $\mathbf{s}$ can be decomposed in the form 
\begin{align*}
\mathbf{s}
  &=\boxplus_{V\in\mathcal{J}}\mathfrak{g}[\mu_{V}].	
\end{align*}
Thus, by Lemma \ref{eq:keyineq1}, we have that 
\begin{align}\label{eq:dWboundtech2}
d_{W}(\nu_n,\mathbf{s})
  &\leq\sum_{V\in\mathcal{J}}d_W(\mu_{V},\mathfrak{g}[\mu_{V}]) 	
\leq\sum_{V\in\mathcal{J}}\mathbbm{d}_{W}(\mu_{V},\mathfrak{g}[\mu_{V}]).
\end{align}
By the Kantorovich-Rubenstein duality, 
\begin{align}\label{eq:dboundliotech2}
\mathbbm{d}_{W}(\mu_{V},\mathfrak{g}[\mu_{V}])
  =\sup_{\substack{h\in\mathcal{C}^1(\R;\R)\\ \|Dh\|_{\infty}\leq 1}}|\langle \mu_{V},h\rangle-\langle \mathfrak{g}[\mu_{V}],h\rangle|.
\end{align}
By Proposition \ref{prop:main}, finding bounds for 
$\langle \mu_{V},h\rangle-\langle \mathfrak{g}[\mu_{V}],h\rangle$
simplifies to bounding from above the term
\begin{align*}
\langle \mathcal{S}_{\boxplus}^{*}[\mu_V],\mathcal{L}_{\boxplus}[Dh]\rangle.
\end{align*}
This task now does not require the use of a superconvergence argument, since the support of $\mu_V$ is contained in the ball of radius $C\mathcal{D}(E_n)$, for some $C>0$. By the Cauchy inversion formula, 
\begin{align*}
h(x)
  =\frac{1}{2\pi\mathbf{i}}\int_{\mathcal{R}}\frac{h(z)}{z-x}dz, \ \ \ \ \ \ \ 
Dh(x)
  =\frac{1}{2\pi\mathbf{i}}\int_{\mathcal{R}}\frac{h(z)}{(z-x)^{2}}dz,	
\end{align*}
where now $\mathcal{R}$ is the rectangle determined by the vertices $(\pm \mathcal{B},\pm 1).$, where $\mathcal{B}:=C\mathcal{D}(E)+1.$ Following a line of reasoning analogous to the proof of Theorem \ref{thm:inhomogeneousBerryEsseen}, we can show that 
\begin{align}\label{eq:technicalone2sasd}
|\langle \mathbf{s},h\rangle-\langle \mu_V,h\rangle|
    &\leq \frac{ \mathcal{B}}{2\pi}\int_{\R_{+}}\sup_{z\in\mathcal{R}}\mathcal{Z}_{\theta,z}d\theta,
\end{align}
where $\mathcal{Z}_{\theta,z}$ is now defined by 
\begin{align*}
\mathcal{Z}_{\theta,z}
  &:=\left|\int_{[-\mathcal{B},\mathcal{B}]^{2}}\left(xg(x)+2(x-z)g(x)g(y)\right) P_{\theta}[\mu_{V}](dx) P_{\theta}[\mu_{V}](dy)\right|.
\end{align*}
As before, we define the non-commutative random variables
\begin{align*}
\xi_{V}^{\theta}
  &:=e^{-\theta}\xi_{V}+\sqrt{1-e^{-2\theta}}s_{V},	
\end{align*}
where $s_{V}$ is distributed according to $\mathfrak{g}[\mu_{V}]$. We can couple the variables $\xi_V$ with semicircular variables $\psi_V$, in such a way that
 \begin{align*}
 \tau[|\xi_V-\psi_V|]	
   &=d_{W}(\mu_V,\mathfrak{g}[\mu_V]).
 \end{align*}
Without loss of generality, we can assume that the $\psi_V$ are free from the $s_V$. Finally, we define
\begin{align*}
\psi_{V}^{\theta}
  &:=e^{-\theta}\psi_{V}+\sqrt{1-e^{-2\theta}}s_{V}.	
\end{align*}
Observe that the distribution of $\psi_V^{\theta}$ is semicircular with mean zero and variance $\tau[\psi_V^2]$. In particular, the law of $\psi_{V}^{\theta}$ coincides with the law of $\psi_{V}^{\infty}$. The centering condition of $\xi_V^{\theta}$ allows us to write 
\begin{align*}
\int_{[-\mathcal{B},\mathcal{B}]}xg(x)P_{\theta}[\mu_{V}](dx)
  &=\tau\left[ \xi_{V}^{\theta}(g\left(\xi_{V}^{\theta}\right)-g (\psi_V^{\theta} ))\right].
\end{align*}
Proceeding as before, after an application of Lemma \ref{lemma:atechone}, we obtain
\begin{align}\label{eq:decompmainone}
\int_{[-\mathcal{B},\mathcal{B}]}xg(x)P_{\theta}[\nu_{n}](dx)
  &=\sum_{1\leq j\leq q-1}\tau\left[\xi_{V}^{\theta}g(\psi_V^{\theta})({\Delta} (\psi_V^{\theta},\xi_{V}^{\theta}-\psi_V^{\theta} )g(\psi_V^{\theta}))^j\right]+\tilde{\mathcal{E}}_{\theta,n},
\end{align}
where
\begin{align*}
\tilde{\mathcal{E}}_{\theta,n}^1(z)
  &:=\tau\left[\xi_{V}^{\theta}g(\xi_V^{\theta})({\Delta}\left(\psi_V^{\theta},\xi_{V}^{\theta}-\psi_{V}^{\theta}\right)g(\psi_V^{\theta}))^q\right].
\end{align*}

Similarly, by first writing 
\begin{align*}
\int_{\R^{2}}\frac{2(x-z)}{(z-x)^{2}(z-y)^{2}}P_{\theta}[\mu_{V}](dx) P_{\theta}[\mu_{V}](dy)	
  &=2\tau[(\xi_{V}^{\theta}-z)g(\xi_{V}^{\theta})]\tau[g(\xi_{V}^{\theta})],
\end{align*}
and then applying Lemma \ref{lemma:atechone}, we obtain
\begin{multline*}
\int_{\R^{2}}\frac{2(x-z)}{(z-x)^{2}(z-y)^{2}}P_{\theta}[\mu_{V}](dx) P_{\theta}[\mu_{V}](dy)\\
\begin{aligned}	
  &=\hat{\mathcal{E}}_{\theta,n}+2\sum_{0\leq j_1,j_2\leq q-1}\tau\left[(\xi_{V}^{\theta}-z)g(\psi_V^{\theta})(\Delta(\psi_V^{\theta},\xi_V^{\theta}-\psi_V^{\theta})g(\psi_V^{\theta}))^{j_1}\right]\\
  &\times \tau\left[g(\psi_V^{\theta})(\Delta(\psi_V^{\theta},\xi_V^{\theta}-\psi_V^{\theta})g(\psi_V^{\theta}))^{j_2}\right],
\end{aligned}
\end{multline*}
where 
\begin{align*}
\hat{\mathcal{E}}_{\theta,n}
&:=2\sum_{0\leq j\leq q-1}\tau\left[(\xi_{V}^{\theta}-z)g(\psi_V^{\theta})(\Delta(\psi_V^{\theta},\xi_V^{\theta}-\psi_V^{\theta})g(\psi_V^{\theta}))^{j}\right]\tau\left[g(\psi_V^{\theta})(\Delta(\psi_V^{\theta},\xi_V^{\theta}-\psi_V^{\theta})g(\psi_V^{\theta}))^{q}\right]\\
  &+2\sum_{0\leq j\leq q-1}\tau\left[(\xi_{V}^{\theta}-z)g(\psi_V^{\theta})(\Delta(\psi_V^{\theta},\xi_V^{\theta}-\psi_V^{\theta})g(\psi_V^{\theta}))^{q}\right]\tau\left[g(\psi_V^{\theta})(\Delta(\psi_V^{\theta},\xi_V^{\theta}-\psi_V^{\theta})g(\psi_V^{\theta}))^{j}\right].
\end{align*}
From here it follows that 
\begin{align*}
\mathcal{Z}_{\theta,z}
	&=\mathcal{Z}_{\theta,z}^1+\mathcal{Z}_{\theta,z}^2+\mathcal{Z}_{\theta,z}^3+\tilde{\mathcal{E}}_{\theta,n}
	+\hat{\mathcal{E}}_{\theta,n},
\end{align*}
where 
\begin{align*}
\mathcal{Z}_{\theta,z}^1
	&:=-\tau\left[\xi_{V}^{\theta}g(\psi_V^{\theta})(\xi_{V}^{\theta}-\psi_V^{\theta} )^2g(\psi_V^{\theta})\right]\\
\mathcal{Z}_{\theta,z}^2
	&:=\sum_{2\leq j\leq q-1}\tau\left[\xi_{V}^{\theta}g(\psi_V^{\theta})({\Delta} (\psi_V^{\theta},\xi_{V}^{\theta}-\psi_V^{\theta} )g(\psi_V^{\theta}))^j\right]\\
\mathcal{Z}_{\theta,z}^3
	&:=2\sum_{\substack{0\leq j_1,j_2\leq q-1\\j_1j_2\neq 0}}\tau\left[(\xi_{V}^{\theta}-z)g(\psi_V^{\theta})(\Delta(\psi_V^{\theta},\xi_V^{\theta}-\psi_V^{\theta})g(\psi_V^{\theta}))^{j_1}\right] \tau\left[g(\psi_V^{\theta})(\Delta(\psi_V^{\theta},\xi_V^{\theta}-\psi_V^{\theta})g(\psi_V^{\theta}))^{j_2}\right].
\end{align*}
Using \eqref{eq:fullexpansion}, we can write 
\begin{align*}
\mathcal{Y}_{\theta,z}^1
  &=-\tau\left[(\xi_{V}^{\theta})^3]\tau[g(\psi_V^{\theta})]\tau[g(\psi_V^{\theta})\right]+2\tau\left[(\xi_{V}^{\theta})^2]\tau[g(\psi_V^{\theta})]\tau[\psi_V^{\theta} g(\psi_V^{\theta})\right],	
\end{align*}
as well as
\begin{align*}
\mathcal{Y}_{\theta,z}^2
	&=\sum_{2\leq j\leq q-1}\sum_{\mathbbm{i}\in\mathcal{I}}\mathfrak{f}_{j,\mathbbm{i}}^1\mathfrak{Q}_{j,\mathbbm{i}}^1(z)\tau\left[\xi_{V}^{\theta}g(\psi_V^{\theta})	\Upsilon_{1,\mathbbm{i}}(\psi_V^{\theta},\xi_{V}^{\theta}
  )\right]\\
  &+\sum_{2\leq j\leq q-1}\sum_{\mathbbm{i}\in\mathcal{I}}\mathfrak{f}_{j,\mathbbm{i}}^2	\mathfrak{Q}_{j,\mathbbm{i}}^2(z)\tau\left[\xi_{V}^{\theta}g(\psi_V^{\theta})\Upsilon_{2,\mathbbm{i}}(\psi_V^{\theta},\xi_{V}^{\theta})\right],
\end{align*}
and 
\begin{align*}
\mathcal{Y}_{\theta,z}^3
	&=2\sum_{\substack{0\leq j_1,j_2\leq q-1\\j_1j_2\neq 0}}\sum_{\mathbbm{i},\mathbbm{j}\in\mathcal{I}}\tau\left[(\xi_{V}^{\theta}-z)g(\psi_V^{\theta})(\mathfrak{f}_{j_1,\mathbbm{i}}^1\mathfrak{Q}_{j_1,\mathbbm{i}}^1(z)\Upsilon_{1,\mathbbm{i}}(\psi_V^{\theta},\xi_V^{\theta})
	+\mathfrak{f}_{j_1,\mathbbm{i}}^2\mathfrak{Q}_{j_1,\mathbbm{i}}^2(z)\Upsilon_{2,\mathbbm{i}}(\psi_V^{\theta},\xi_V^{\theta}))\right]\\
	&\times \tau\left[g(\psi_V^{\theta})\mathfrak{f}_{j_2,\mathbbm{j}}^1\mathfrak{Q}_{j_2,\mathbbm{j}}^1(z)\Upsilon_{1,\mathbbm{j}}(\psi_V^{\theta},\xi_V^{\theta})
	+\mathfrak{f}_{j_2,\mathbbm{j}}^2\mathfrak{Q}_{j_2,\mathbbm{j}}^2(z)\Upsilon_{2,\mathbbm{j}}(\psi_V^{\theta},\xi_V^{\theta}))\right],
\end{align*}
where $\Upsilon_{j,\mathbbm{i}}$ are given by \eqref{eq:Upsilondef}. Next we observe that by Proposition \ref{prop:krewerasprop}
\begin{align*}
\tau\left[\xi_{V}^{\theta}g(\psi_{V}^{\theta})	\Upsilon_{1,\mathbbm{i}}(\psi_{V}^{\theta},\zeta_{V}^{\theta}))
  )\right]
    &=\sum_{\pi\in NC(n)}\kappa_{\pi}[\xi_{V}^{\theta},(\xi_{V}^{\theta})^{i_2},\dots, (\xi_V^{\theta})^{i_{|\mathbbm{i}|}}]\\
    &\times\tau_{K(\pi)}[g(\psi_{V}^{\theta})(\psi_{V}^{\theta})^{i_1},(\psi_{V}^{\theta})^{i_3},\dots, (\psi_{V}^{\theta})^{i_{|\mathbbm{i}|-1}}].	
\end{align*}
By Lemma \ref{lem:momentmatchingPUrel}, the term $\kappa_{\pi}[\xi_{V}^{\theta},(\xi_{V}^{\theta})^{i_2},\dots, (\xi_V^{\theta})^{i_{|\mathbbm{i}|}}]$ doesn't depend on $\theta$. Moreover, using the fact that $P_{\theta}[\psi_V^{\theta}]$ is equal in law to $\psi_V^{\theta}$. From here we obtain 
\begin{align*}
\tau\left[\xi_{V}^{\theta}g(\psi_{V}^{\theta})	\Upsilon_{1,\mathbbm{i}}(\psi_{V}^{\theta},\zeta_{V}^{\theta}))
  )\right]
    &=\tau\left[\xi_{V}^{\infty}g(\psi_{V}^{\infty})	\Upsilon_{1,\mathbbm{i}}(\psi_{V}^{\infty},\zeta_{V}^{\infty}))
  )\right].	
\end{align*} 
Proceeding similarly, 
\begin{align*}
\tau\left[\xi_{V}^{\theta}g(\psi_V^{\theta})\Upsilon_{2,\mathbbm{i}}(\psi_V^{\theta},\xi_{V}^{\theta})\right]
  &=\tau\left[\xi_{V}^{\infty}g(\psi_V^{\infty})\Upsilon_{2,\mathbbm{i}}(\psi_V^{\infty},\xi_{V}^{\infty})\right],
\end{align*}
for $i=2$. By a similar argument, we can deduce that the identity is valid for $i=1,3$ as well. By Stein's identity, we have that $\mathcal{Y}_{\infty,z}=0$, thus implying
\begin{align*}
|\mathcal{Y}_{\theta,z}|
	&=|\mathcal{Y}_{\theta,z}-\mathcal{Y}_{\infty,z}|=|\tilde{\mathcal{E}}_{\theta,n}-\tilde{\mathcal{E}}_{\infty,n}|
	+|\hat{\mathcal{E}}_{\theta,n}-\hat{\mathcal{E}}_{\infty,n}|,	
\end{align*}
The two terms in the right-hand side are handled in a similar fashion, so we will restrict ourselves to the estimation of $\tilde{\mathcal{E}}_{\theta,n}^1$. First we write 
\begin{align*}
|\tilde{\mathcal{E}}_{\theta,n}^1|
  &\leq |\tau\left[(\xi_{V}^{\theta}-\xi_{V}^{\infty})g(\psi_V^{\theta})({\Delta}\left(\psi_V^{\theta},\xi_{V}^{\theta}\right)g(\psi_V^{\theta}))^q\right]|\\
  &+ |\tau\left[\xi_{V}^{\theta}(g(\psi_V^{\theta})-g(\psi_V^{\infty}))({\Delta}\left(\psi_V^{\theta},\xi_{V}^{\theta}\right)g(\psi_V^{\theta}))^q\right]|\\
  &+ |\tau\left[\xi_{V}^{\theta}g(\psi_V^{\theta})(({\Delta}\left(\psi_V^{\theta},\xi_{V}^{\theta}\right)g(\psi_V^{\theta}))^q
  -({\Delta}\left(\psi_V^{\infty},\xi_{V}^{\theta}\right)g(\psi_V^{\theta}))^q\right]|\\
&+ |\tau\left[\xi_{V}^{\theta}g(\psi_V^{\theta})(({\Delta}\left(\psi_V^{\theta},\xi_{V}^{\theta}\right)g(\psi_V^{\theta}))^q
  -({\Delta}\left(\psi_V^{\theta},\xi_{V}^{\infty}\right)g(\psi_V^{\theta}))^q\right]|\\
  &+ |\tau\left[\xi_{V}^{\theta}g(\psi_V^{\theta})(({\Delta}\left(\psi_V^{\theta},\xi_{V}^{\theta}\right)g(\psi_V^{\theta}))^q
  -({\Delta}\left(\psi_V^{\theta},\xi_{V}^{\theta}\right)g(\psi_V^{\infty}))^q\right]|.
\end{align*}
By the boundedness of $g$, and the fact that 
\begin{align*}
|\tilde{\mathcal{E}}_{\theta,n}^1|
  &\leq Ce^{-\theta}\tau\left[(|\psi_V|+|\xi_V|)^{q+1}\right].
\end{align*}
From here we can easily check that  
\begin{align*}
|\tilde{\mathcal{E}}_{\theta,n}^1-\tilde{\mathcal{E}}_{\infty,n}^1|
  &\leq Ce^{-\theta}\tau\left[ |\xi_V|^{q+1}\right],
\end{align*}
for a possibly different constant $C>0$. Proceeding similarly, we obtain 
\begin{align*}
|\hat{\mathcal{E}}_{\theta,n}^1-\hat{\mathcal{E}}_{\infty,n}^1|
  &\leq Ce^{-\theta}\tau\left[ |\xi_V|^{q+1}\right].
\end{align*}
By \eqref{eq:technicalone2sasd}, we then deduce that 
\begin{align*}
|\langle \mathbf{s},h\rangle-\langle \mu_V,h\rangle|
    &\leq C (\mathcal{D}(E_n)+1)\tau\left[ |\xi_V|^{q+1}\right].
\end{align*}
Summing over $V\in\mathcal{J}$ and using \eqref{eq:dWboundtech2} and \eqref{eq:dboundliotech2} and the  H\"older inequality, we get
\begin{align*}
d_{W}(\nu_n,\mathbf{s})
  &\leq C (\mathcal{D}(E_n)+1)\sum_{V\in\mathcal{J}}\tau\left[ |\xi_V|^{q+1}\right]	
  \leq C (\mathcal{D}(E_n)+1)^{q+1}\sum_{k=1}^nm_{q+1}[\mu_{k,n}]	,
\end{align*}
as required.

\section{Technical lemmas}\label{sec:technicallemas}
The decomposition below gives a Taylor-type expansion for non-commutative variables when the underlying function under consideration is $g(x):=(z-x)^2$, for some $z$ lying in the upper half-plane. Recall that $\mathfrak{s}$ is defined to be the symmetrization operator acting over the set of polynomials over non-commutative variables.
\begin{lemma}\label{lemma:atechone}
For non-commutative variables $a,r$, define 
\begin{align*}
\Delta\left(a,r\right)
  &:=2\mathfrak{s}[(z-a)r]-r^2,
\end{align*}
where $\mathfrak{s}$ denotes the symmetrization operator. Then, for all $q\geq 1$,
\begin{align*}
g\left(a+r\right)
  &=g\left(a+r\right)\left(\Delta\left(a,r\right) g (a)\right)^{q}+\sum_{j=0}^{q-1}g\left(a\right)\left(\Delta\left(a,r\right)g\left(a\right)\right)^j,
\end{align*}	
	
\end{lemma}
\begin{proof}
The case $q=1$ follows from a direct computation. For the general case, we use assume the validity of the identity for $q$, and then observe that 
\begin{align*}
g\left(a+r\right)
  &=g\left(a+r\right)\left(\Delta\left(a,r\right) g (a)\right)^{q}+\sum_{j=0}^{q-1}g\left(a\right)\left(\Delta\left(a,r\right)g\left(a\right)\right)^j.
\end{align*}
The induction hypothesis for $q=1$, yields
\begin{align*}
g\left(a+r\right)
  &=g\left(a+r\right)\Delta\left(a,r\right) g (a)+ g\left(a\right).
\end{align*}
A combination of these two identities gives the desired result.
\end{proof}

\noindent\textbf{Acknowledgements}\\
We thank Octavio Arizmendi for helpful guidance in the elaboration of this paper. Arturo Jaramillo Gil was supported by the grants  CB-2017-2018-A1-S-9764 and  CBF2023-2024-2088.

\bibliographystyle{plain}
\bibliography{bibliography}

\end{document}